\documentclass{amsart}
\usepackage{graphicx} 

\usepackage{amsmath}%
\usepackage{amsthm}
\usepackage{amsfonts}%
\usepackage{amssymb}%
\usepackage{graphicx}
\usepackage{tikz-cd}
\usetikzlibrary{cd}
\usepackage{bm}
\usepackage{comment}
\usepackage{todonotes}

\usepackage{mathrsfs}

\DeclareFontFamily{OT1}{pzc}{}
\DeclareFontShape{OT1}{pzc}{m}{it}{<-> s * [1.10] pzcmi7t}{}
\DeclareMathAlphabet{\mathpzc}{OT1}{pzc}{m}{it}

\newtheorem{theorem}{Theorem}[section]

\newtheorem{corollary}[theorem]{Corollary}

\newtheorem{definition}[theorem]{Definition}

\newtheorem{lemma}[theorem]{Lemma}

\newtheorem{problem}[theorem]{Problem}
\newtheorem{proposition}[theorem]{Proposition}
\newtheorem{remark}[theorem]{Remark}

\numberwithin{equation}{section}

\def\XXint#1#2#3{{\setbox0=\hbox{$#1{#2#3}{\int}$}
\vcenter{\hbox{$#2#3$}}\kern-.5\wd0}}

\newcommand{\R}{\mathbb{R}}
\renewcommand{\P}{\mathbb{P}}
\newcommand{\Q}{\mathbb{Q}}
\newcommand{\N}{\mathbb{N}}

\newcommand{\Ha}{\mathcal{H}}
\newcommand{\leb}{\mathcal{L}}

\newcommand{\spt}{\operatorname{spt}}

\newcommand{\Mod}{\operatorname{Mod}}
\newcommand{\dist}{\operatorname{dist}}
\newcommand{\diam}{\operatorname{diam}}

\newcommand{\lip}{\operatorname{lip}}
\newcommand{\Lip}{\operatorname{Lip}}
\newcommand{\LIP}{\operatorname{LIP}}
\newcommand{\im}{\operatorname{im}}
\newcommand{\gap}{\operatorname{gap}}
\newcommand{\dom}{\operatorname{dom}}
\newcommand{\Fr}{\operatorname{Fr}}

\newcommand{\conv}{\operatorname{conv}}
\newcommand{\sh}{\operatorname{sh}}

\newcommand{\ud}{\mathrm {d}}
\newcommand{\id}{\mathrm {id}}

\newcommand{\inv}{^{-1}}

\newcommand{\Tan}{\operatorname{Tan}}

\usepackage{stmaryrd}

\title[Embeddability and rectifiability of LDS]{Embeddability and rectifiability of Lipschitz differentiability spaces}
\author{Ivan Caamano, Sylvester Eriksson-Bique, Elefterios Soultanis}
\date{\today}

\thanks{IC is partially supported by grant PID2022-138758NB-I00 (Spain).  SEB is supported by the Research
Council of Finland  grant no. 354241. ES was partially supported by Research
Council of Finland grant no. 355122. The work in this manuscript was also partially supported by the Simons Foundation grant (award no. SFI-MPS-T-Institutes-00010825) and from State Treasury funds as part of a task commissioned by the Minister of Science and Higher Education under the project “Organization of the Simons Semesters at the Banach Center - New Energies in 2026-2028” (agreement no. MNiSW/2025/DAP/491).}

\begin{document}

\maketitle

\begin{abstract}
We prove that Lipschitz differentiability spaces which bi-Lipschitz embed into an RNP-space are countably rectifiable. In contrast to earlier methods of Cheeger and Kleiner, our approach does not rely on differentiating RNP-targets, and uses instead decomposability bundles and a careful blow-up analysis. We also present decomposability bundles in a way which avoids the mention of Alberti representations and generalizes the approach of Alberti--Marchese \cite{almar16} to measures in RNP-spaces. 

We moreover study fragment-wise differentiability into RNP-targets, give a new $\Lip-\lip$-type characterization of RNP-differentiability spaces, and address a question of Le Donne asking for a characterization of spaces $(X,\mu)\subset\ell^2$ whose Gromov--Hausdorff tangents are Hausdorff limits of $r\inv (X-x)$ in $\ell^2$ as $r\to 0$.

\end{abstract}

\section{Introduction}

\subsection{Background}
The seminal work of Cheeger \cite{che99} on the differentiability of Lipschitz functions on PI-spaces gave birth to what are now called \emph{Lipschitz differentiability spaces} (LDS), where the conclusion of Rademacher's theorem is valid with respect to suitable charts. The theory of LDS's was pioneered, along with Cheeger, by numerous researchers including e.g. Bate, Kleiner, Li, Schioppa, Speight \cite{bate12diff,bate2013differentiability,bate17,sch16b,sch16a,che-kle-sch16}. A necessary (but not sufficient) condition for being an LDS is the existence of a decomposition of the measure into curve-fragments, known as Alberti representations after the influential work \cite{almar16}. Alberti representations are closely connected to Weaver derivations and metric 1-currents \cite{sch16,sch16b}, and give rise to a (weaker) notion of fragment-wise differentiability \cite{tsd24}.

It turns out that Lipschitz differentiability and fragment-wise differentiability behave very differently if we consider Lipschitz maps into infinite dimensional Banach spaces: while Lipschitz maps from an LDS to an RNP-Banach space are always fragment-wise differentiable (Theorem \ref{thm:fragmentwise}), an example of Schioppa \cite{schioppadiff,waldbate} shows that this is not true of full differentiability. The question which targets a given LDS differentiates -- and how this is reflected in its geometry -- is currently poorly understood, apart from the results of second author and Bate--Li \cite{bateli,sebgafa} that differentiating all RNP-targets (i.e. being RNP-LDS) is equivalent to PI-rectifiability (being covered by countably many subsets of PI-spaces). An important reason for interest in such questions comes from (non-)embeddings of metric spaces. A conceptually simple non-embedding principle is present already in \cite[Theorem 14.2]{che99} and distilled in \cite[Theorem 1.2 and Corollary 1.3]{terirect}: an LDS which both differentiates, and (bi-Lipschitz) embeds into, a given Banach space must be countably rectifiable; see also \cite[Corollary 3.2]{lytopen} for a closely related statement. Thus, PI-spaces embedding into an RNP-space are countably rectifiable. 

In this paper, we show that rectifiability can -- rather surprisingly -- be deduced directly, without appealing to RNP-differentiability, but using instead the infinitesimal geometry of LDS's and the geometric properties of RNP-spaces. This generalizes several rectifiability results in the literature and demonstrates a novel strategy for analyzing subsets of RNP-spaces, which may be of independent interest. We furthermore develop fragment-wise differentiation into RNP-targets (establishing Theorem \ref{thm:fragmentwise} mentioned above) and decomposability bundles in an infinite dimensional setting, generalizing results of \cite{almar16}. We will next give the precise formulation of the main theorem.

\subsection{Rectifiability}
A \emph{Cheeger chart} $(U,\varphi)$ in a metric measure space $(X,d,\mu)$ consists of a Borel set $U\subset X$ with $\mu(U)>0$, and a Lipschitz map $\varphi:X\to \R^n$ ($n$ is called the dimension of the chart) with the following property: given any Lipschitz function $f:X\to \R$, for $\mu$-a.e. $x\in U$ there exists a unique $\ud_xf\in(\R^n)^*$ so that 
\begin{align}\label{eq:LDS-diff}
	\limsup_{y\to x}\frac{|f(y)-f(x)-\ud_xf(\varphi(y)-\varphi(x))|}{d(y,x)}=0.
\end{align}
The linear map $\ud_xf$ is called the differential of $f$ at $x$ with respect to $(U,\varphi)$. 

\begin{definition}\label{def:LDS}
    A metric measure space $(X,d,\mu)$ is called a Lipschitz differentiability space (LDS) if it can be covered by countably many Cheeger charts (of possibly varying dimension).
\end{definition}

We say that a metric measure space $(X,\mu)$ is countably rectifiable, if there is countable collection of sets $X_i\subset X$ with 
\begin{align*}
	\mu\Big(X\setminus \bigcup_iX_i\Big)=0
\end{align*}
so that, for each $i\in\N$, $\mu|_{X_i}\ll\Ha^{n_i}$ and $X_i$ is bi-Lipschitz equivalent to a subset of $\R^{n_i}$.

A Banach space $Y$ is said to satisfy RNP (the Radon Nikodym Property) if every Lipschitz function $f:[0,1]\to Y$ is differentiable almost everywhere.

\begin{theorem}\label{thm:LDS-in-RNP}
Suppose $(X,\mu)$ is a Lipschitz differentiability space and $\iota:X\to Y$ is a bi-Lipschitz map into a Banach space $Y$ with the RNP. Then $(X,\mu)$ is countably rectifiable.
\end{theorem}

With a stronger assumption of RNP-differentiability this was established in  \cite[Theorem 1.6]{cheegerkleiner}. That result rests on the ability to differentiate the biLipschitz embedding itself to find almost everywhere unique Hausdorff tangents that are planes. In our case, we also show that at almost every point Hausdorff tangents are planes, but do this via a much more delicate general study of the infinitesimal structure coming from a fragment-wise differential structure. We first show that at almost every point the space infinitesimally admits a splitting $Z\times H$ with $H$ a finite dimensional subspace, similar to the results in \cite{davideb20}. We then show that $Z$ must be trivial. The difficulty here is that the existence of non-trivial extra factors $Z$ alone is insufficient to derive a contradiction by constructing a non-differentiable function. Our argument relies -- in addition to having non-trivial extra factors of tangents -- on the inclusion of the space in an RNP-space, which yields that the rescalings are Hausdorff-close to products, cf. Lemma \ref{lem:e-close-to-product}. This subtle point is important in the proof. Finally, we employ the construction of Schioppa \cite{sch16a} of tile functions together with this approximate splititng.

We obtain immediately a novel obstruction to embeddability.
\begin{corollary}
Suppose $(X,\mu)$ is an LDS which fails to differentiate some RNP-Banach space. Then $(X,\mu)$ does not admit a bi-Lipschitz embedding into any RNP-Banach space.
\end{corollary}
In particular, the examples of Schioppa \cite{schioppadiff} and Bate--Wald \cite{waldbate} do not embed into any RNP-space.

This corollary has implications to the well-known Lang-Plaut problem \cite[Question 2.4]{langplaut}, which asks whether every doubling subset of $\ell_2$ biLipschitz embeds into $\R^n$ for some $n$ depending on the doubling constant. One plausible proof strategy has been suggested: Find a subset $(X,\mu)\subset \ell_2$ which is an LDS, but which is not rectifiable. Such a space could also not embed into $\R^n$ for any finite $n$ due to \cite[Theorem 14.2]{che99}. Our main result, Theorem \ref{thm:LDS-in-RNP}, shows that no such example exists. Thus, differentiability cannot serve as an obstruction to finite dimensional embeddability for subsets of $\ell_2$.

\subsection{Decomposability bundle and fragment-wise differential}
An important technical tool in our approach is the decomposability bundle of a measure. The notion as well as the terminology originates in \cite{almar16}, where a definition in terms of Alberti representations is given. Below we generalize this notion to infinite dimensional spaces, and give an alternative definition which does not use Alberti representations; see also Section \ref{sec:decomp-bundle}.

\begin{definition}\label{def:decomp-bundle}
Let $\mu$ be a Radon measure in an RNP-Banach space $Y$. The decomposability bundle $x\mapsto T_\mu(x)$ is the $\mu$-a.e. minimal bundle with the following property: there exists a $\mu$-null set $N\subset X$ such that
\[
\gamma_t' \in T_\mu({\gamma_t}) \text{ a.e. } t\in \dom(\gamma)
\]
for all $\gamma\in \Fr(Y)$ with $|\gamma\inv(N)|=0$.
\end{definition}

The decomposability bundle is closely related to the fragment-wise differentiable structure. In particular, we establish the existence of the decomposability bundle whenever $(X,\mu)$ admits a fragment-wise differentiable structure in Proposition \ref{prop:linear-part-of-blowup}. We also show its existence in a wide class of  infinite dimensional situations, cf. Theorem \ref{thm:decomp-bundle}. These results include all uniformly convex Banach spaces.

We now contrast the rigidity of LDS subsets of RNP-Banach spaces with fragment-wise differentiability.
\begin{theorem}\label{thm:fragmentwise}
Let $(U,\varphi)$ be an $n$-dimensional fragment-wise chart in a metric measure space $(X,\mu)$, and let $Y$ be an RNP-Banach space. Then for any $f\in\LIP(X,Y)$ and $\mu$-a.e. $x\in U$, there exists a unique $\ud_xf\in \leb(\R^n,Y)$ such that 
\begin{align*}
(f\circ\gamma)'_t= \ud_{\gamma_t}f((\varphi\circ\gamma)_t')\quad\textrm{a.e. }t\in \gamma^{-1} (U)
\end{align*}
for $\Mod_\infty$-a.e. $\gamma\in\Fr(X)$. 
\end{theorem}

We also prove a fragment-wise $w^*$-differentiation Theorem in Section \ref{sec:w-star}. Using Theorem \ref{thm:fragmentwise} we give a characterization of RNP-LDS in terms of the pointwise Lipschitz constant, in the spirit of \cite[Theorem 1.5.]{tsd24} and \cite[Theorem 1.19]{sch16b}, \cite[Theorem 1.4.]{bkt19}. We formulate Theorem \ref{thm:RNP-LDS-char} below in terms of arbitrary Banach spaces $Y$. In the statement, we say that $(X,\mu)$ is a $Y$-LDS if every $f\in \LIP(X,Y)$ is differentiable in the sense of \eqref{eq:LDS-diff} $\mu$-a.e..

\begin{theorem}\label{thm:RNP-LDS-char}
Let $Y$ be an RNP-Banach space, and $(X,\mu)$ a metric measure space. Then the following are equivalent.
\begin{itemize}
\item[(i)] $X$ is $Y$-LDS;
\item[(ii)] $\Lip f=|Df|_\ast$ $\mu$-a.e. on $X$ for all $f\in \LIP(X,Y)$;
\item[(iii)] There is a collection $\omega =\{ \omega_x \}_{x\in X}$ of moduli of continuity such that 
$$\Lip f\le \omega(|Df|_\ast)\quad \mu \mbox{-a.e.}$$ 
for all $f\in \LIP(X,Y)$.
\end{itemize}
\end{theorem}

\begin{remark} We make a few remarks on the connections of differentiability and the conditions in the previous theorem to rectifiability.
In \cite[Theorem 1.2]{terirect} metric differentiability of maps into arbitrary targets was shown to imply rectifiability. In contrast, the condition $\Lip f\lesssim |Df|_\ast$ for Lipschitz maps into arbitrary targets is not sufficient to imply rectifiability. Indeed, any PI-space will satisfy this condition by an argument in \cite{cheegerkleiner}. For scalar valued $f$ this condition holds whenever $(X,\mu)$ is a LDS, cf. \cite[Theorem 1.15.]{sch16b} or \cite[Theorem 1.5.]{tsd24}. 
\end{remark}

\subsection{Applications and further discussion}

As an application of our results we address the following question posed by Le Donne \cite[Question 1.8.]{ledonne}: characterize doubling spaces $(X,\mu)\subset \ell^2$ for which the GH-tangents at $\mu$-a.e. $x\in X$ are realized as Hausdorff limits of $r\inv(X-x)$ as $r\to 0$. Below we give an answer to this problem, when additionally the Hausdorff tangent is equal to the decomposability bundle. This gives a simple sufficient condition for the existence of Hausdorff limits in terms of being LDS.

\begin{theorem}\label{thm:le donne}
    Suppose $Y$ is an RNP-Banach space, $(X,\mu)\subset Y$ and $\mu$ vanishes on porous sets of $X$. Then the following are equivalent.
\begin{itemize}
    \item[(i)] $(X,\mu)$ is an LDS;
    \item[(ii)] $(r\inv X,x)$ pGH-converges to $T_\mu(x)$ for $\mu$-a.e. $x\in X$;
    \item[(iii)] $r\inv(X-x)$ Hausdorff converges to $T_\mu(x)$ for $\mu$-a.e. $x\in X$.
\end{itemize}
\end{theorem}

In general, if $r\inv(X-x)$ Hausdorff converges to a limit $H_x$ as $r\to 0$, then one can show that for $\mu$-a.e. $x\in X$ the tangent $H_x$ is necessarily a subspace. This follows by combining three facts, and we only sketch this argument here since we do not directly need it. First, every limit $H_x$ satisfies $tH_x=H_x$ for $t>0$, since we obtain $tH_x$ as the limit of $(t^{-1}r)^{-1}(X-x)$. This implies that if $y\in H_x$ is non-zero, then the tangent at $y$ of $H_x$ contains a line. The Preiss phenomenon states that a tangent of a tangent is a tangent, and by uniqueness, therefore, if $H_x$ is not a singleton, it must contain a line. If $\mu$ is doubling, this follows from \cite[Theorem 1.1.]{ledonne}, and for measures vanishing on porous subsets, this is shown in the third authors and Cakovic's work \cite{soultaniscakovic}. This argument can be iterated: if $L$ is a maximal dimensional subspace contained in $H_x$, then if there is some $y$ not in this subspace, then at $y$ the span of $y$ and $L$ is contained in the Hausdorff tangent of $H_x$. This contradicts the maximality of $L$ and thus $L=H_x$.

\begin{remark}
It is possible that $r\inv(X-x)$ Hausdorff converges to a subspace $H_x$ which is strictly larger than $T_\mu(x)$ $\mu$-a.e. $x$. In this case $(X,\mu)$ is purely unrectifiable, and is not an LDS. For example if $\mu$ is a singular doubling measure on $X=\R$ (constructed e.g. in , then $T_\mu(x)=\{0\}$ $\mu$-a.e. $x$ but $r\inv(X-x)=\R$ for all $r$.
\end{remark}

We end this introduction with an open problem on differentiability. 
\begin{problem}
    Is every LDS an $\ell_2$-LDS?
\end{problem}

All known examples \cite{waldbate,schioppadiff} satisfy $\ell_2$-differentiability. In fact, in these examples, $\ell_2$-differentiability is used to prove Lipschitz differentiability. The examples in \cite{waldbate} show that for every $p>2$ there exists an LDS which is not an $\ell_p$-LDS.  It seems that some of the constructions of non-differentiable functions used in this paper may be useful for this problem. Our methods do not, however, distinguish between $\ell_p$ and $\ell_2$. We also note that the problem seems closely related to the Lang-Plaut problem mentioned above.

\section{Preliminaries}

\subsection{Notation and conventions}

A metric measure space consists of a complete separable metric space $X=(X,d)$ equipped with a Borel regular measure $\mu$ which is finite on balls. We often denote a metric measure space by $(X,\mu)$. We use the standard notation $B(x,r)$ and $\bar{B}(x,r)$ for the open and closed balls that are  centered at $x\in X$ and radius $r>0$, and $\lambda B(x,r)=B(x,\lambda r)$ for $\lambda >0$. A measure $\mu$ is called doubling if there exists a constant $D>1$ s.t. $\mu(2B)\leq D\mu(B)$ for all balls $B=B(x,r)\subset X$. For non-empty $A\subset X$ and $x\in X$, write $\dist_A(x)=\inf_{a\in A} d(a,x)$. A metric space $X$ is called metrically doubling, if there exists a constant $N\in \mathbb{N}$ so that for every ball $B(x,R)\subset X$ there are $N$ points $x_1,\dots, x_N\in X$ s.t. $B(x,R)\subset \bigcup_{i=1}^N B(x_i, R/2)$.

Let $(V,\Vert\cdot\Vert )$ be a Banach space. For a mapping $f:X\rightarrow V$ we denote
\begin{align*}
&\LIP (f):=\sup_{x,y\in X}\frac{\Vert f(x)-f(y)\Vert}{d(x,y)},\\
&\Lip f(x):=\lim_{r\to 0^+}\sup_{y\in B(x,r)\setminus\{ x\}}\frac{\Vert f(x)-f(y)\Vert}{r},\quad\mbox{and}\\
&\Lip_af(x):=\lim_{r\to 0^+}\LIP(f|_{B(x,r)})\quad \mbox{for}\; x\in X.
\end{align*}
If $x\in X$ is an isolated point then we let $\Lip f(x)=0=\Lip_af(x)$. We say that $f$ is Lipschitz if $\LIP (f)<\infty $, and denote the collection of all Lipschitz maps $X\to V$ by $\LIP (X,V)$. Recall that a Banach space $V$ has the Radon-Nikodým property, RNP for short, if every Lipschitz mapping $f:[a,b]\rightarrow V$ is differentiable almost everywhere, in the sense that for a.e. $t\in [a,b]$ there exists a linear map $\ud_tf:\R\rightarrow V$ so that 
$\Lip (f-\ud_tf)(t)=0$.

\subsection{Curve fragments and $*$-upper gradients}

\subsubsection*{Curve fragments.} Denote by $\Fr (X)$ the set of all curve fragments in $X$, i.e., bi-Lipschitz mappings $\gamma :\dom (\gamma )\rightarrow X$ from a compact set $\dom (\gamma )\subset \R$. Here, by bi-Lipschitz we mean that there exists a constant $L>0$ so that $L^{-1}|b-a|\leq d(\gamma_a,\gamma_b)\leq L|b-a|$ for all $a,b\in\dom (\gamma )$. Notice that we are adopting the notation $\gamma_t=\gamma (t)$ when evaluating a curve fragment.

Let $I_\gamma :=\conv (\dom (\gamma ))=[a,b]$, then $I_\gamma\setminus\dom (\gamma )$ is an open set in $\R$ and thus it is a countable union of open intervals $(a_i,b_i)$, $i\in I\subset \N$. We call such intervals gaps of $\gamma$ and define the gap magnitude of $\gamma$ as
$$\gap (\gamma ):=\sum_{i\in I} d(\gamma_{a_i},\gamma_{b_i}).$$
We refer to \cite[Section 2.2]{tsd24} for the definition of the semihull $\sh (X)$ that allows to extend $\gamma$ to a Lipschitz curve $\overline{\gamma}:I_\gamma\rightarrow \sh (X)$. This extension might not be bi-Lipschitz, but since it is Lipschitz by \cite[Proposition 1]{kir94} it has a metric speed a.e.  and then, by restricting back to $\dom (\gamma )$, we obtain that
$$\vert\gamma^\prime_t\vert :=\lim_{\dom (\gamma )\ni s\to t}\frac{d(\gamma_s,\gamma_t)}{|s-t|}\quad\mbox{exists for a.e. }\; t\in \dom(\gamma)$$

If $Y$ is an RNP-Banach space and $\gamma\in\Fr(Y)$, then 
\begin{align}\label{eq:frag-derivative}
\gamma'(t)=\lim_{s\to t}\frac{\gamma(s)-\gamma(t)}{s-t}\quad\mbox{exists for a.e. }\; t\in \dom(\gamma).
\end{align}
Indeed, the curve $\gamma:I_\gamma\to Y$ obtained by extrapolating linearly in $I_\gamma\setminus\dom(\gamma)$ is differentiable a.e. on $I_\gamma$, and for all 
Lebesgue density points of $\dom(\gamma)$ where the derivative exists, it agrees with the limit in \eqref{eq:frag-derivative}.  
Moreover, we have that
\begin{align*}
    |\gamma_t'|=\|\gamma'(t)\|_Y\quad\mbox{a.e. }\;t\in \dom(\gamma).
\end{align*}

Similarly, if $f\in\LIP(X,Y)$ and $\gamma\in \Fr(X)$, then 
\begin{equation}\label{eq:diffcomposition}
(f\circ\gamma )^\prime (t)=\lim_{\dom (\gamma \ni s\to t}\frac{f(\gamma_s)-f(\gamma_t)}{|s-t|}
\end{equation}
exists for almost every $t\in\dom (\gamma )$.

\subsubsection*{*-Upper gradients.} For a curve fragment $\gamma\in \Fr (X)$ we define the line integral of a Borel function $\rho :X\rightarrow [0,\infty )$ by
$$\int_\gamma \rho\, ds :=\int_{\dom (\gamma )}\rho (\gamma_t )|\gamma_t^\prime |\, dt=\int_{\mathrm{Im} (\gamma )}\rho\, d\mathcal H^1.$$
For a family $\Gamma\subset \Fr (X)$ of curve fragments we define its $\infty$-modulus as
$$\Mod_\infty (\Gamma ):=\inf\left\{ \Vert \rho\Vert_{L^\infty (X)}:\int_{\gamma}\rho\, ds \geq 1\mbox{ for all }\gamma\in\Gamma \right\}.$$
We say that a property $P$ holds for $\Mod_\infty$-a.e. curve fragment if the collection of curve fragments $\Gamma_P$, which fail $P$, satisfies $\Mod_\infty(\Gamma_P)=0$. By a classical argument, see e.g. \cite[Lemma 2.7]{tsd24}, this is equivalent to the existence of a null set $N\subset X$ for which each $\gamma\in \Gamma_P$ satisfies $\int_\gamma 1_N ds >0$.
For a map $f:X\rightarrow V$ and $\gamma\in\Fr (X)$ denote $\displaystyle \mathrm{osc}_\gamma f:=\sup_{x,y\in \mathrm{Im}(\gamma )}\Vert f(x)-f(y)\Vert $.
\begin{definition}
Let $f\in\LIP (X;V)$, we say that a Borel function $\rho :X\rightarrow [0,\infty )$ is a $*$-upper gradient of $f$ if
\begin{equation}\label{eq:*ug}
\mathrm{osc}_\gamma f\leq \int_\gamma \rho\, ds +\LIP (f)\gap (\gamma )
\end{equation}
for every $\gamma\in \Fr (X)$. If the family $\Gamma_0\subset \Fr (X)$ of curve fragments such that \eqref{eq:*ug} does not hold satisfies $\Mod_\infty (\Gamma_0 )=0$ then we say that $\rho$ is a weak $*$-upper gradient.
\end{definition}
We recall that, due to \cite[Proposition 2.10]{tsd24}, every Lipschitz map $f\in \LIP (X)$ admits a minimal $*$-upper gradient $|Df|_*$. We will see in Lemma \ref{lem:ugscalarized} that this is also the case for Lipschitz mappings into a Banach space, and we will adopt the same notation, $|Df|_*$, for the minimal $*$-upper gradient of a map $f:X\rightarrow V$. Moreover, the same arguments as in \cite[Remark 2.9]{tsd24} yield that any weak $*$-upper gradient of a map in $\LIP (X,V)$ admits a $\mu$ representative that is a $*$-upper gradient.

\subsection{Fragment-wise charts}
\subsubsection*{Alberti representations and fragment-wise charts}  An Alberti representation $\mathcal A=\{\mu_\gamma,\mathbb P\}$ of a measure $\mu$ on $X$ consists of a finite positive measure $\mathbb P$ on $\Fr (X)$ and a family $\{\mu_\gamma\}$ of probability measures on $X$ such that
\begin{itemize}
	\item[(a)] $\mu_\gamma\ll\mathcal H^1|_{{\rm Im}(\gamma)}$ $\mathbb P$-a.e. $\gamma$;
	\item[(b)] $\gamma\mapsto \mu_\gamma(B)$ is $\mathbb P$-measurable and $\displaystyle \mu(B)=\int\mu_\gamma(B)\ud \mathbb P(\gamma)$ for every Borel $B\subset X$.
\end{itemize}

Given $z\in S^{n-1}$, $\varepsilon>0$, define the cone in the $z$-direction of width $\varepsilon$ as
$$C(z,\varepsilon):=\{p\in \R^n: z\cdot p> (1-\varepsilon)|p| \}$$
and we say that cones $C_1,\dots ,C_n$ are independent if $v_1,\dots ,v_n$ are linearly independent for any choice of $v_i\in C_i$.

Let  $\varphi\in \LIP(X,\R^n)$; we say that the Alberti representation $\mathcal A$ is in the $\varphi$-direction of $C(z,\varepsilon)$ if $(\varphi\circ\gamma)'(t)\in C(z,\varepsilon)$ a.e. $t\in {\rm dom}(\gamma)$ for $\mathbb P$-a.e. $\gamma\in {\rm Fr}(X)$.
We say that $\mathcal A_1,\dots ,\mathcal A_n$ are $n$ $\varphi$-independent Alberti representations if there are independent cones $C_1,\dots ,C_n$ so that $\mathcal A_i$ is in the $\varphi$-direction of $C_i$ for each $i=1,\dots ,n$.
  We refer to \cite{bate2013differentiability,sch16b}  for more information on Alberti representations.

\begin{definition}\label{def:fwchart}
We call the pair $(U,\varphi )$ a fragmwnt-wise chart of dimension $n$ in $X$ if $U\subset X$ is a Borel set of positive measure, $\varphi :X\rightarrow \R^n$ is Lipschitz and the following two conditions hold:
\begin{enumerate}
    \item[$(i)$] \textbf{\em Independence:} $\mu|_U$ admits $n$ $\varphi$-independent Alberti representations.
    \item[$(ii)$] \textbf{\em Maximality:} If $V\subset U$ and $\mu|_V$ admits $k$ $\psi$-independent Alberti representations, for some Lipschitz $\psi:X\to \R^k$, then $k\leq n$.
\end{enumerate}
\end{definition}

Existence of a maximal collection of independent Alberti representations, in the sense of Definition \ref{def:fwchart}, holds for any metric measure space with finite Hausdorff dimension \cite[Theorem 5.5]{tsd24}. Moreover, this is enough to construct a differential along curves of Lipschitz maps, see \cite[Theorem 1.2]{tsd24} for the scalar case. Furthermore, a description of the $*$-upper gradients through a seminorm of this differential is given in \cite{tsd24}, which also gives a characterization of Lipschitz differentiability spaces in terms of $*$-upper gradients \cite[Theorem 1.5]{tsd24}. These should be contrasted with an ealier description of a differential in terms of directional derivatives in \cite[Corollary 2.13]{bate12diff}. We will see in Section \ref{sec:fw} that these ideas extend to RNP-valued mappings; to this end, we formally introduce the notion of fragment-wise differential for RNP-valued mappings.
\begin{definition}
Let $(U,\varphi )$ be a chart in a metric measure space $(X,d,\mu )$, $V$ a RNP Banach space and $f:X\rightarrow V$. We say that $\ud_xf\in \leb(\R^n,V)$ is a fragment-wise differential of $f$ at a point $x\in U$, with respect to $(U,\varphi )$, if
\begin{align*}
(f\circ\gamma)'_t= \ud_{\gamma_t}f((\varphi\circ\gamma)_t')\quad\textrm{a.e. }t\in \gamma^{-1} (U)
\end{align*}
for $\Mod_\infty$-a.e. $\gamma\in\Fr(X)$. 
\end{definition}

We will use the following alternative characterization of fragment-wise chart later in the paper, which also may be of independent interest. We emphasize that the statement only involves real-valued Lipschitz functions -- later in Section \ref{sec:fw} we will study the existence of the fragment-wise  differential for Banach valued mappings.

\begin{lemma}\label{lem:chartchar} Let $U\subset X$ be Borel and $\varphi:X\to \R^n$ be Lipschitz. Then $(U,\varphi)$ is a fragment-wise chart if and only if for every $f:X\to \R$ there exists a a.e. unique fragment-wise differential of $f$ with respect to $(U,\varphi)$.
\end{lemma}
\begin{proof}
    The existence and uniqueness of a fragment-wise differential for $f:X\to \R$ with respect to a fragment-wise chart $(U,\varphi)$ was given in \cite[Theorem 1.2]{tsd24}. We show now the converse. 

    To show independence, we consider the Borel functional $\Phi:(\R^n)^*\times X\to [0,\infty]$ given in \cite[Proposition 4.5, Proposition 4.6, Proposition 5.2]{tsd24} that satisfies two properties:
    \begin{enumerate}
        \item If $x\to \lambda_x$ is Borel and $\Phi(\lambda_x, x)=0$ for all $x\in V$ for some Borel subset $V\subset X$, then $\lambda_{\gamma_t}((\varphi \circ \gamma)'_t)=0$ for $\Mod_\infty$-a.e. $\gamma\in \Fr(X)$ and a.e. $t\in \gamma^{-1}(V)$.
        \item For the set $I:=\{x\in U: \inf_{|\lambda|=1} \Phi(\lambda,x)>0\}$ we have that $\mu|_I$ has $n$-independent Alberti representations. 
    \end{enumerate}

    We claim that $\mu(U\setminus I)=0$. If not, then by a Borel selection argument, there exists a Borel map $x\to \lambda_x$ s.t. $\Phi(\lambda_x, x)=0$ and $\lambda_x\neq 0$ for all $x\in U\setminus I$. By (1), for every Lipschitz $f:X\to \R$ we have that  $\ud_x f + c\lambda_x \chi_{U\setminus I}$ is a fragment-wise differential of $f$ for all $c\in \R$. This is a contradiction to uniqueness, and thus $\mu(U\setminus I)=0.$ Therefore, $\mu|_U$ has $n$ independent Alberti representations.

    Next, we argue maximality. Suppose that $V\subset U$ is a Borel set and $\psi:X\to \R^k$ is a Lipschitz map s.t. $\mu|_V$ has $k$ $\psi$-independent Alberti representations. Let $C_1,\dots, C_k$ be the independent cones in the definition of $\psi$-independent Alberti representations. Let $\ud \psi_i$  be the fragment-wise differentials with respect to $\varphi$. For $\mu$-a.e. $x\in V$, the image $T(x)={\rm span}\{(\ud_x \psi_1(e_i), \cdots, \ud_x \psi_k(e_i)) : i=1,\dots, n\}$ is an at most $n$ dimensional subspace of $\R^k$. From the definition of a fragment-wise differential, we see $(\psi \circ \gamma)'_t\in T(\gamma_t)$ for $\Mod_\infty$-a.e. $\gamma\in \Fr(X)$ and a.e. $t\in \dom(\gamma)$. In particular, by \cite[Lemma 2.7]{tsd24} there exists a null set $N$, s.t. $(\psi \circ \gamma)'_t\in T(\gamma_t)$ for all $\gamma\in \Fr(X)$ and a.e. $t\in \dom(\gamma)$ s.t. $\gamma_t\not\in N$.  If $k>n$, there must exist an index $i\in \{1,\dots, k\}$ and a positive measure subset $K$ of $V$ s.t. $\mu|_K$ has an Alberti represention $\{\mu_\gamma,\mathbb{P}\}$ in $\psi$-direction $C_i$ and $C_i\cap T(x)=0$. Then, since $(\psi\circ \gamma)'_t\not\in T(\gamma_t)$ for a.e. $t\in \dom(\gamma)$ and $\mathbb{P}$-a.e. $\gamma\in \Fr(X)$, we have $\mu_\gamma(K)=0$ for $\mathbb{P}$-a.e. $\gamma\in \Fr(X)$. Thus $\mu(K)=0$ by the definition of an Alberti representation, which is a contradiction.
\end{proof}

\subsection{Lipschitz Differentiability Spaces.} Recall from the Section 1.2 that a Cheeger chart (of dimension $n$) in a metric measure space $(X,\mu)$ is a pair $(U,\varphi)$ where $U\subset X$ is a Borel set with $\mu(U)>0$ and $\varphi\in\LIP(X,\R^n)$ satisfies the following: given any $f\in \LIP(X)$, for $\mu$-a.e. $x\in U$ there exists a unique $\ud_xf\in (\R^n)^*$ such that \eqref{eq:LDS-diff} holds. Recall also the definition of a Lipschitz Differentiability space in Definition \ref{def:LDS}.

We refer to \cite{sch16b,kei02,bate12diff,bate2013differentiability} for some background on the many structural properties enjoyed by LDS's. We mention here that if $(X,\mu)$ is an LDS, then $\mu$ vanishes on porous sets \cite[Theorem 2.4]{bate2013differentiability}, and is pointwise doubling  \cite[Corollary 2.6]{bate2013differentiability}. Recall that a set $S\subset X$ is $\eta$-porous at $x\in S$, if there exists a sequence $x_i\to x$ with $$d(x_i,S)\ge\eta d(x_i,x),$$ and porous at $x$ if it is $\eta$-porous at $x$ for some $\eta>0$; we call $S$ porous if it is porous at $x$ for all $x\in S$. 

Furthermore, we say that $\mu$ is point-wise doubling at $x\in \operatorname{spt}\mu$ if
\begin{align*}
    \limsup_{r\to 0}\frac{\mu(B(x,2r))}{\mu(B(x,r))}<\infty.
\end{align*}
A metric measure space $(X,\mu)$ is called pointwise doubling if $\mu$ is point-wise doubling at $x$ for $\mu$-a.e. $x$.

\subsection{Tangents}

Let $Z$ be a metric space and $A_i\subset Z$, $i\in\N\cup\{\infty\}$. We say that $A_i$ boundedly Haurdorff-converges to $A_\infty$ (denoted $A_i\stackrel{H}{\longrightarrow} A_\infty$) if
\begin{align*}
    \lim_{i\to\infty}\sup\{\dist(a,A_i): a\in A_\infty\cap B\}&=0\quad\textrm{ and } \\
     \lim_{i\to\infty}\sup\{\dist(b,A_\infty): b\in A_i\cap B\}&=0
\end{align*}
for all bounded sets $B\subset Z$. Clearly it is sufficient to consider (closed) balls in place of $B$. 
\begin{remark}
The convergence defined above is better known as Attouch--Wets convergence, see e.g. \cite{attwets}, and the quantity appearing in the above limits is known as the excess:
\begin{equation}\label{eq:excess}
e_B(A,C):=\sup\{\dist(a,C):\ a\in A\cap B\}.
\end{equation}
We will occasionally use this notation in the sequel. Note that bounded Hausdorff convergence can be metrized, see e.g. \cite[Definition 2.15]{Bate22} and the discussion in \cite[Section 3]{Badger-Krandel-Vellis26}, specifically \cite[Lemma 3.1]{Badger-Krandel-Vellis26}.
\end{remark}
We say that a sequence of pointed metric measure spaces $(X_i,\mu_i, x_i)$ converges to a pointed metric measure space $(X_\infty,\mu_\infty,x_\infty)$ in the \emph{pointed measured Gromov--Hausdorff} sense, denoted $X_i\stackrel{pmGH}{\longrightarrow}X_\infty$, if there exist pointed isometric embeddings $\iota_i:(X_i,x_i)\to (Z,z)$ ($i\in\N\cup\{\infty\}$) so that $\iota_i(X_i)\stackrel{H}{\longrightarrow}\iota_\infty(X_\infty)$ and moreover $\iota_{i\ast}\mu_i\rightharpoonup \iota_{\infty\ast}\mu_\infty$ in duality with the space $C_{bbs}(Z)$ of bounded continuous functions on $Z$ with bounded support. If we drop the condition on the measures, the pointed metric spaces $(X_i,x_i)$ are said to converge to $(X_\infty,x_\infty)$ in the \emph{pointed Gromov--Hausdorff} sense, denoted $(X_i,x_i)\stackrel{pGH}{\longrightarrow}(X_\infty,x_\infty)$. 

If $f_i:X_i\to V$ are $L$-Lipschitz maps into a finite dimensional Banach space for some $L>0$, we say that $(f_i)$ converges to $f_\infty:X_\infty\to V$ in the pointed measured Gromov--Hausdorff sense, denoted $f_i\stackrel{pmGH}{\longrightarrow}f_\infty$, if $X_i\stackrel{pmGH}{\longrightarrow}X_\infty$ and the isometries $\iota_i$ above satisfy
\begin{align*}
f_i(x_i)\to f_\infty(x)\quad\mathrm{whenever}\quad x_i\in X_i,\quad \iota_i(x_i)\stackrel{i\to\infty}{\longrightarrow}\iota_\infty(x)
\end{align*}
for all $x\in X_\infty$. Using the separability of $X_\infty$ it is direct to show that if $X_i\stackrel{pmGH}{\longrightarrow}X_\infty$, then any sequence of $L$-Lipschitz functions $f_i:X_i\to \R$ (sending basepoints to zero) has a convergent subsequence.

Given $x\in X$ and $r>0$, we denote 
\begin{align*}
    T_{x,r}X=\Big(r\inv X,\frac{\mu}{c_{\mu}(x,r)},x\Big),\quad c_\mu(x,r)=\int\Big(1-\frac{d(x,y)}{r}\Big)_+\ud\mu(y).
\end{align*}
\begin{definition}\label{def:tangent}
Let $(X,\mu)$ be a metric measure space. A pointed metric measure space $Y=(Y,\nu,o)$ is a pointed measured Gromov--Hausdorff (pmGH) tangent of $X$ at $x$, if there exists a sequence $r_i\to 0$ so that $T_{x,r_i}X\stackrel{i\to\infty}{\longrightarrow}Y$ in the pmGH sense. The collection of pmGH-tangent of $X$ at $x$ is denoted $\Tan(X,\mu,x)$. 
\end{definition}
Since we only consider pmGH-tangents in this paper, we omit the subscript in the notation $\Tan(X,\mu,x)$.  Given an $L$-Lispchitz function $f:X\to \R$ and $Y\in \Tan(X,\mu,x)$ arising as the pmGH-limit of $T_{x,r_i}X$, we call $f_\infty:Y\to\R$ a \emph{blow-up} of $f$ (at $x$), if the sequence $f_i:=\frac{f-f(x)}{r_i}:T_{x,r_i}X\to \R$ pmGH-converges to $f_\infty$.

It is well-known that, if $(X_i,\mu_i,x_i)$ is a sequence where each measure is $C$-doubling, then $(X_i,\mu_i,x_i)$ pmGH-converges up to a subsequence to some limit $(Y,\nu,o)$ which is also $C$-doubling; see e.g. the discussion in \cite[Lemma 3.32.]{gigconv}. In particular, $\Tan(X,\mu,x)$ is non-empty for every $x\in X$ if $\mu$ is a doubling measure. This observation is also true for measures which vanish on porous sets, see \cite[Theorem 1.5, Proposition 5.3]{soultaniscakovic} as well as the discussions in  \cite[Remark 3.13]{sch16a} and \cite[Remark 2.23]{CKS}. We record this in the following proposition.

\begin{proposition}\label{prop:porous-tangent}
Let $(X,\mu)$ be a metric measure space where $\mu$ vanishes on porous sets. Then $\Tan(X,\mu,x)\ne\varnothing$ and every element of $\Tan(X,\mu,x)$ is doubling, for $\mu$-a.e. $x\in X$.
\end{proposition}
\begin{proof}
    This follows from \cite[Theorem 1.5]{soultaniscakovic} and \cite[Lemma 5.6]{Bate22}.
\end{proof}

\section{Fragment-wise differentiation}\label{sec:fw}

Recall that $X$ is separable and let $f:X\to V$ be Lipschitz. We will be showing the existence of a minimal $\ast$-upper gradient for $f$. This will be done via the notion of a norming set.

If $A\subset V$, we say that $S\subset \overline{B_{V^*}}$ is a norming set for $A$ if for all $a\in A$ there exists $s\in S$ s.t. $\|a\|=|\langle s,a\rangle|$. Here, $\overline{B_{V^*}}$ is the closed unit ball in $V^*$, and we denote the dual pairing by $\langle s,a\rangle:=s(a)$. When $S$ is countable, it is also refered to as a norming sequence for $A$.

Our proofs will make use of the existence of a countable norming set of functionals for a dense subset of the difference set 
    $$f(X)-f(X):=\{ f(x)-f(y):x,y\in X\}.$$
    Indeed, by the continuity of $f$ and the separability of $X$ we have that there exists a countable collection $\{ v_i\}_{i\in\N}\subset V$ dense in $f(X)-f(X)$. 
    By Hahn-Banach Theorem we can then consider a countable norming set of functionals $\{ v_i^*\}_{i\in\N}\subset \overline{B_{V^*}}$, for which $|\langle v_i^*,v_i\rangle|=\Vert v_i\Vert$ for all $i\in\N$. 
\begin{lemma}\label{lem:ugscalarized}
    Let $V$ a Banach space and $f\in \LIP (X,V)$ and $f:X\rightarrow V$. Then there exists a minimal $*$-upper gradient of $f$, denoted by $|Df|_*$. Moreover, for any countable $Z\subset \overline{B_{V^*}}$ that is a norming set for a dense subset of $f(X)-f(X)$ we have
    \[
    |Df|_*=\sup_{v^*\in Z}\{ |D\langle v^*,f\rangle |_*\}.
    \]

    Moreover, if $f,g:X\to V$ are Lipschitz and $\lambda\in \R$, then the following two properties hold a.e.:
    \begin{enumerate}
        \item Homogeneity:
        \[
        |D(\lambda f)|_* = |\lambda| |Df|_*.
        \]
        \item Subadditivity: 
        \[
        |D(f+g)|_* \leq |D(f)|_* + |D(g)|_*.
        \]
    \end{enumerate}
\end{lemma}
    We point out that the above lemma will be used several times for norming sequences of a dense subspaces containing $f(X)$. In this case, the linear structure of the subspace also yields the norming property for the difference set.
\begin{proof}
   Let $Z=\{ v_i^*\}_{i\in\N}\subset B_{V^*}$ be a norming sequence for a collection $\{ v_i\}_{i\in\N}$ dense in $f(X)-f(X)$. The existence of such a set $Z$ was explained before the proof.  Then for each $x,y\in X$ there is a subsequence $\{ v_{i_j}\}_{j\in\N}$ that converges to $f(x)-f(y)$ and
    \begin{eqnarray}
\Vert f(x)-f(y)\Vert &=&\lim_{j\to\infty }\Vert v_{i_j}\Vert =\lim_{j\to\infty }\vert\langle v_{i_j}^*,v_{i_j}\rangle\vert \nonumber \\
&\leq & \limsup_{j\to\infty }\left(\left\vert \langle v_{i_j}^*,v_{i_j}+f(y)-f(x)\rangle\right\vert +\left\vert \langle v_{i_j}^*,f(x)-f(y)\rangle\right\vert\right)\nonumber \\
&\leq &\limsup_{j\to\infty }\left( \Vert v_{i_j}-f(x)+f(y)\Vert+\left\vert \langle v_{i_j}^*,f(x)-f(y)\rangle\right\vert\right)\nonumber \\
&=&\limsup_{j\to\infty}\left\vert \langle v_{i_j}^*,f(x)-f(y)\rangle\right\vert  \label{eq:hahnbanach1}.
\end{eqnarray}
    On the other hand $\langle v_i^*,f\rangle :X\rightarrow \R$ is Lipschitz for all $i\in\N$ and thus by \cite[Proposition 2.10]{tsd24} each $\langle v_i^*,f\rangle$ has a minimal $*$-upper gradient $|D\langle v_i^*,f\rangle|_*$. Let
    \begin{equation}\label{eq:supviug}
        \rho^*:=\sup_{i\in\N} |D\langle v_i^*,f\rangle|_*.
    \end{equation}
    Notice that $\Lip f$ is a $*$-upper gradient for each $\langle v_i^*,f\rangle$, so $\rho^*\in L^\infty (X)$. Now for a curve fragment $\gamma\in \Fr (X)$ by \eqref{eq:hahnbanach1} and the definition of $|D\langle v_i^*,f\rangle|_*$
     we have for each $x,y\in \mathrm{Im} (\gamma )$
     $$ \Vert f(x)-f(y)\Vert \leq \limsup_{j\to\infty}\left\vert \langle v_{i_j}^*,f(x)-f(y)\rangle\right\vert \leq \int_\gamma \rho^* \, ds +\gap (\gamma )\LIP (f).$$
     This concludes that $\rho^*$ is a $*$-upper gradient of $f$. Now suppose that $\rho\in L^\infty (X)$ is another $*$-upper gradient of $f$. In particular it is a $*$-upper gradient of $\langle v_i^*,f\rangle$ for all $i\in\N$. This yields $|D\langle v_i^*,f\rangle|_*\leq \rho$ a.e. for all $i\in\N$ and thus $\rho^*\leq \rho$ a.e. We thus have $|Df|_*=\rho^*$.

     From Formula \eqref{eq:supviug} for the minimal $\ast$-upper gradient, we get subadditivity and homogeneity by same properties for weak $\ast$-upper gradients for real-valued functions, which follows directly from the definition. Indeed, it is easy to note that $|\lambda||Df|_\ast$ is an upper gradient for $\lambda f$ for every $f:X\to \R$, and $|Df|_\ast + |Dg|_\ast$ is an upper gradient for $f+g$. From these the claims follow by uniqueness of $\ast$-upper gradients.

     \end{proof}

\subsection{Proof of Theorem \ref{thm:fragmentwise}.}
Let $(U,\varphi )$ be a fragment-wise chart of dimension $n$ in $X$ and $f:X\rightarrow V$ Lipschitz, where $V$ is a Banach space satisfying the Radon-Nikodým property. We first recall that for each $\gamma\in \Fr (X)$ we have the existence of 
$$ (f\circ\gamma)_t'=\lim_{{\rm dom}(\gamma)\ni t'\to t}\frac{f(\gamma_{t'})-f(\gamma_t)}{t'-t}\quad \textrm{a.e. }t\in \dom(\gamma).$$ 
Moreover, $f$ is Lipschitz, thus measurable and continuous, which implies that the set $W=\overline{\rm span}\{ f(U) \}\subset V$ is separable. By Hahn-Banach's theorem there exist $v_i^*\in V^*$, $i\in\N$ with $\Vert v_i^*\Vert\leq 1$ so that
\begin{equation}\label{eq:hahnbanach}\Vert v\Vert =\sup_{i\in\N}\langle v_i^*,v\rangle\quad\text{for all } v\in W.
\end{equation}
By the fragment-wise differentiability of scalar valued maps \cite[Theorem 1.2]{tsd24}, for each $v^*\in V^*$ there exists a unique Borel map $L_{(\cdot )}[v^*]:U\rightarrow (\R^n)^*$, given as $L_x[v^*]=d_x\langle v^*, f\rangle$, which satisfies
\begin{equation}\label{eq:fwdiffcomposedwithfunctional}
	L_{\gamma_t}[v^*]((\varphi \circ\gamma)^\prime_t )=(\langle v^*,f\rangle \circ\gamma )^\prime_t=\langle v^*,(f\circ\gamma )^\prime_t\rangle .
\end{equation}
at a.e. $t\in \gamma^{-1}(U)$ for $\Mod_\infty$-a.e. $\gamma\in \Fr (X)$. Let $Z:=\overline{\rm span}\{ v_i^*:i\in\N\}\subset V^*$ and $Z_0$ a countable vector space over $\Q$ dense in $Z$. For $x\in U$, such that it exists, consider $L(x):Z_0\rightarrow (\R^n)^*$ defined as $v^*\mapsto L_x[v^*]$. We first prove that $L(x)$ is linear and bounded for almost every $x\in U$.

\textit{Linearity.} Let $v^*_{1},v^*_{2}\in Z_0$, $a_1,a_2\in\Q$, $\gamma\in \Fr(X)$ and $t\in \gamma^{-1}(U)$ such that \eqref{eq:fwdiffcomposedwithfunctional} holds for $v^*_{1}$ and $v^*_{2}$. Then
\begin{align*}
	\langle a_1v^*_{1}+a_2v^*_{2}, (f\circ\gamma )_t^\prime\rangle &= 
	a_1\langle v^*_{1}, (f\circ\gamma )_t^\prime\rangle +
	a_2\langle v^*_{2}, (f\circ\gamma )_t^\prime\rangle 
	\\
	&=(a_1L_{\gamma_t}(v^*_{1})+ L_{\gamma_t}(v^*_{2}))((\varphi\circ\gamma )_t^\prime ),
\end{align*}
and since $x\mapsto L_x[v^*]$ is unique for each $v^*\in V^*$ then $L_{\gamma_t}(a_1v^*_{1}+a_2v^*_{2})=a_1L_{\gamma_t}(v^*_1)+ L_{\gamma_t}(v^*_2)$, naturally extending to arbitrary linear combinations and thus proving linearity in $Z_0$.

\textit{Boundedness.} For a.e. $x\in U$, we can equip $(\R^n)^\ast$ with the norm $|\cdot|_x$ given by \cite[Theorem 1.2.]{tsd24}, for which we have
\[
|L_x(v^\ast)|_x=|d_x\langle v^*, f\rangle|_x = |D\langle v^*, f\rangle|_\ast(x)\leq \LIP[\langle v^*, f\rangle]\leq \|v^\ast\|\LIP[f].
\]

Since $L(x)$ is a bounded linear operator $L(x):Z_0\rightarrow ((\R^n )^*,|\cdot|_x )$ we can then consider its extension $Z\rightarrow ((\R^n )^*,|\cdot|_x)$ and the corresponding adjoint operator $d_xf:(\R^n,\vert\cdot\vert )\rightarrow Z^*\subset V^{**}$, i.e., it satisfies $\langle d_xf(v),v^*\rangle =L_x[v^*](v)$ for all $v\in \R^n$. Here $|\cdot |$ denotes the dual norm of $|\cdot|_x$. In particular, when $v=(\varphi\circ\gamma)^\prime_t$ for some $\gamma\in\Fr (X)$ satisfying \eqref{eq:fwdiffcomposedwithfunctional} for all $v^*\in Z_0$ (still $\Mod_\infty$-a.e. curve fragment, since $Z_0$ is countable), the continuity of the extension of $L(x)$ yields then for every $v^*\in Z$
\begin{equation}\label{eq:adjoint-fwdiff}
\langle d_{\gamma_t}f((\varphi\circ\gamma )_t^\prime ),v^*\rangle =L(\gamma_t)[v^*]((\varphi\circ\gamma )_t^\prime )=\langle v^*, (f\circ \gamma)_t^\prime  \rangle ,\quad \mbox{a.e. } t\in \gamma^{-1}(U).
\end{equation}
Notice that we considered $x\in U\setminus N$, where $\mu (N)=0$, in order to infer boundedness of $v^*\mapsto L_x[v^*]$, but this is not restrictive to obtain \eqref{eq:adjoint-fwdiff} for a.e. $t\in\gamma^{-1}(U)$ due to \cite[Lemma 2.7]{tsd24}, which implies $\Gamma_N^*$ has null $\infty$-modulus.

By the definition of $W$ we have that $(f\circ\gamma )^\prime_t\in W$, which is normed by  $B_Z$ due to \eqref{eq:hahnbanach} and the definition of $Z$, so we can  identify $(f\circ\gamma )^\prime_t\in W\hookrightarrow Z^*$. On the other hand, we also have $d_xf$ maps into $Z^*$, and then, since $B_{Z_0}$ is norming for $Z^*$ due to the density of $Z_0$ in $Z$, 
\begin{align*}\Vert d_xf((\varphi\circ\gamma )^\prime_t)-(f\circ\gamma )^\prime_t\Vert =\sup_{v^*\in B_{Z_0}}| \langle d_xf((\varphi\circ\gamma )^\prime_t)-(f\circ\gamma )^\prime_t,v^*\rangle | =0
\end{align*}
for a.e. $t\in\gamma^{-1}(U)$. It follows that for vectors of the form $v=(\varphi\circ\gamma)^\prime_t$ we have $d_xf(v)=(f\circ\gamma)^\prime_t\in W$. 
Moreover, since $W$ is closed and for a.e. $x\in U$, $\{ (\varphi\circ\gamma)^\prime_t/|(\varphi\circ\gamma)^\prime_t|:\gamma\in \Fr (X), \gamma_t=x,\ \gamma\notin \Gamma_0\}$ is dense in $S^{n-1}$ for any $\Mod_\infty$-null family $\Gamma_0$ (cf. \cite[Theorem 6.4]{tsd24}) we have that $\operatorname{Im}(\ud_xf)\subset W\subset V$. \qed

\subsection{Weak$^*$ fragment-wise differentials}\label{sec:w-star}
In this section we will address an analogue of Theorem \ref{thm:fragmentwise} for maps $f:X\rightarrow V$ where $V$ may not have the RNP, but is dual to a separable Banach space $Y$, i.e., $Y$ is separable and $V=Y^*$. Recall that, in this case, Lipschitz mappings $g:I\subset\R \to V$ are weak$^*$ differentiable a.e. (see \cite[Theorem 3.5]{amb-kir00}). We recall here the definition of weak$^*$ differentiability.
\begin{definition}\label{def:w-star}
    Let $I\subset \R$, $V=Y^*$ a Banach space and $g:I\rightarrow V$. We say $g$ is weak$^*$ differentiable at $t\in I$ if there exists a linear map $w\ud_tg:I\rightarrow V$ such that
    \[
    \lim_{s\to t}\left\langle y, \frac{ g(s)-g(t)-w\ud_tg(s-t) }{|s-t|}\right\rangle
     =0\quad\mbox{ for all }y\in Y.
     \]
     In particular, there exists a vector $g_t^{(*)}\in V$ such that
     \[
     \langle y,g\rangle_t^\prime =\langle y, g^{(*)}_t\rangle\mbox{ for all }y\in Y.
     \]
\end{definition}
 In our setting, for a Lipschitz map $f:X\rightarrow V$ and $\gamma\in \Fr (X)$, for each $y\in Y$ we can differentiate tha map $\langle f,y\rangle\circ\gamma  :I\rightarrow \R $ a.e., and since $f\circ\gamma$ is Lipschitz, the existence of a weak$^*$ derivative a.e.  yields
 \begin{equation}\label{eq:w-starderivative}
     (\langle y,f\rangle\circ\gamma )_t^\prime =( \langle y,(f\circ\gamma ) \rangle)_t^\prime  =\langle y,(f\circ\gamma )_t^{(*)} \rangle \quad \mbox{ for a.e. }t\in\dom (\gamma ).
 \end{equation}
 Notice that, although the null set in $\dom (\gamma )$ is independent of $y$ for the weak$^*$ differential, this is not the case for the differential of $\langle y,f\rangle\circ \gamma $.
 We now present an analogue of Theorem \ref{thm:fragmentwise} using weak$^*$ differentiation, which also serves as a definition of fragment-wise weak$^*$ differentials.
\begin{proposition}\label{prop:w-star}
    Let $(U,\varphi)$ be an $n$-dimensional fragment-wise chart in a metric measure space $(X,\mu)$, and let $V$ be dual to a separable Banach space $Y$. Then for any $f\in\LIP(X,V)$ and $\mu$-a.e. $x\in U$, there exists a unique $w\ud_xf\in \leb(\R^n,V)$ such that 
\begin{align*}
(f\circ\gamma)^{(*)}_t= w\ud_{\gamma_t}f((\varphi\circ\gamma)_t')\quad\textrm{a.e. }t\in \gamma^{-1} (U)
\end{align*}
for $\Mod_\infty$-a.e. $\gamma\in\Fr(X)$. We call $w\ud_xf$ the fragment-wise weak$^*$ differential of $f$ at $x$.
\end{proposition}
\begin{proof}
	Let $D\subset Y$ a countable dense vector space over $\Q$. By \eqref{eq:w-starderivative} and the fragment-wise differentiability of real-valued Lipschitz functions \cite[Theorem 1.2]{tsd24} we have that for each $y\in D$ there exists  $L_{(\cdot )}[y]\in (\R^n )^*$ such that
    \begin{equation}\label{eq:w-starfw}
    L_{\gamma_t}[y]((\varphi\circ\gamma )^\prime_t)=(\langle y,f\rangle \circ\gamma )^\prime_t =\langle y,(f\circ\gamma )^{(*)}\rangle 
    \end{equation}
    for $\Mod_\infty$-a.e. $\gamma\in \Fr (X)$ and a.e. $t\in \gamma^{-1}(U)$. 
    Similar arguments as in the proof of Theorem \ref{thm:fragmentwise} yield that for a.e. $x\in U$ the map $y\mapsto L_x[y]$ is linear and bounded with respect to $|\cdot|_x$. Indeed, these follow identically using $y\in D$ as functionals instead of $v^*$. Now, for a.e. $x\in U$ we can extend $L_x$ to a map $L_x:Y\to (\R^n)^*$ that is linear and bounded.

Now, for each $x\in U$ such that $y\mapsto L_x[y]$ is linear and bounded with respect to $|\cdot|_x$, consider $w\ud_xf:(\R^n,|\cdot |)\rightarrow V$, where $|\cdot|$ is the dual norm of $|\cdot|_x$, the adjoint operator of the extension $Y\ni y\mapsto L_x[y]$. Then, we obtain
     \[ 
     \langle y, w\ud_{\gamma_t}f((\varphi\circ\gamma )^\prime_t)= L_{\gamma_t}[y]((\varphi\circ\gamma )^\prime_t) =\langle y,(f\circ\gamma )^{(*)}_t\rangle 
     \]
     for $\Mod_\infty \mbox{-a.e. } \gamma\in \Fr (X), \mbox{a.e. }t\in\gamma^{-1 }(U)$. This holds first for $y\in D$, but density and boundedness imply that holds for all $y\in Y$. Notice that \cite[Lemma 2.7]{tsd24} is applied here so that the above holds for a.e. $t\in \gamma^{-1}(U)$, as the existence of the adjoint operator only holds for a.e. $x\in U$.
     Taking supremum over all $y\in B_Y$ we then get that $df:x\mapsto w\ud_xf$ is a fragment-wise differential of $f$.
\end{proof}

\subsection{Proof of Theorem \ref{thm:RNP-LDS-char}.}
\noindent$(i)\Rightarrow (ii)\quad$ Fix $(U,\varphi )$ a RNP-Cheeger chart and let $f\in \LIP (X,V)$, which by hypothesis admits a differential $d_xf:\R^n\rightarrow V$ satisfying
$$\Lip f(x)=\limsup_{y\to x}\frac{\Vert \ud_xf(\varphi (y)-\varphi (x))\|}{d(x,y)}$$
for a.e. $x\in U$. By \cite[Lemma 2.1]{bate2013differentiability} we have that
$$\Vert \xi \Vert^*_x:=\Lip (\xi \circ \varphi )(x)$$
is a norm in $(\R^n)^*$ for a.e. $x\in U$ and thus we consider $\Vert\cdot\Vert_x$ to be its dual norm. Consider then
\begin{align*}
	\Vert \ud_xf\Vert_{\rm op}&:=\sup \{ \Vert \ud_xf(v)\Vert:\Vert v\Vert_x\leq 1,
\end{align*}
and so we have $\Lip f(x)\leq \Vert d_xf\Vert_{\rm op}$. Thus, in order to infer $(ii)$ we may now prove that $\Vert \ud_xf\Vert_{\rm op}=|Df|_*(x)$. This holds true for real-valued Lipschitz functions due to \cite[Theorem 1.2]{tsd24}. Thus, for a countable family $\{ v_i^*\}_{i\in\N}$ we have simultaneously  that for almost every $x\in U$ $\Vert d_x\langle v_i^*,f\rangle \Vert_{\rm op}=|D\langle v_i^*,f\rangle |_*(x)$ for all $i\in\N$. By choosing this countable collection as in \eqref{eq:hahnbanach}, i.e. norming $W=\overline{\mathrm{span}}\{ f(U)\}$, we can apply Lemma \ref{lem:ugscalarized} to get
$$\sup_{i\in\N}\Vert d_x\langle v_i^*,f\rangle \Vert_{\rm op}=\sup_{i\in\N}|D\langle v_i^*,f\rangle |_*(x)=|Df|_*(x).$$
It is clear that $\displaystyle \sup_{i\in\N}\Vert \ud_x\langle v_i^*,f\rangle \Vert_{\rm op}\leq \Vert \ud_xf\Vert_{\rm op}$. For the reverse inequality consider $v\in\R^n$ such that $\Vert v\Vert_x\leq 1$, and recall from the proof of Theorem \ref{thm:fragmentwise} that $\ud_xf(v)\in W$ is normed by $\{ v_i^*\}_{i\in\N}$, thus
$$
\Vert \ud_xf(v)\Vert =\sup_{i\in\N}|\langle v_i^*,\ud_xf(v)\rangle|= \sup_{i\in\N}|\ud_x\langle v_i^*,f\rangle (v)|\leq \sup_{ i\in\N}\Vert \ud_x\langle v_i^*,f\rangle \Vert_{\rm op}.
$$
Then taking supremum over all such $v$ we have 
$$\Vert \ud_xf\Vert_{\rm op}\leq \displaystyle\sup_{ i\in\N}\Vert \ud_x\langle v_i^*,f\rangle \Vert_{\rm op}=|Df|_*(x)$$ for $\mu$-a.e. $x$.

\noindent$(ii)\Rightarrow (iii)\quad$ This is trivial as one can choose $\omega(t)=t$.

\noindent$(iii)\Rightarrow (i)\quad$ Since $(iii)$ holds, in particular, for scalar valued Lipschitz functions, it follows from \cite[Theorem 1.5]{tsd24} that $X$ is a Lipschitz differentiability space. Consider then $(U,\varphi )$ a (real-valued) Cheeger chart and let $f\in \LIP (X;V)$ for a Banach space $V$ with the RNP. We recall that there is a separable subspace $W\subset V$ with the property that $\ud_xf(\R^n)\subset W$ $\mu$-a.e. $x\in U$ for every $f\in \LIP(X,V)$.  

Consider the Lipschitz map $g:=(f,\varphi):X\to V\times \R^n$. Consider the subspace $S\subset \mathcal B(V\times \R^n,V)$ spanned by linear maps $\Phi_L(v,z)=v+L(z):V\times \R^n\to V$ where $L$ ranges over the (separable) space $\mathcal B(\R^n,W)$. Notice  that for such $L$ we have $\Phi_L\circ g=f+L\circ\varphi :X\rightarrow V$ is Lipschitz, so by (iii) there exists a $\mu$-null set $N_{L}\subset X$ so that 
\begin{equation}\label{eq:ugphil}
\Lip (\Phi_L\circ g)(x)\le \omega(|D \Phi_L \circ g|_*(x))
\end{equation}
whenever $x\notin N_L$. Notice also that for $\mu$-a.e. $x\in X$ the map $\lambda \mapsto \omega ( |D \lambda \circ g|_*)$ is uniformly continuous on dense subsets of $S$ due to the subadditivity properties in Lemma \ref{lem:ugscalarized}. It is also easy to check that $\lambda\mapsto \Lip (\lambda\circ g)(x) $ is a seminorm on $S$.

On the other hand, by the boundedness of $L$ and the definition of fragment-wise differential we have
\begin{equation}\label{eq:ugphil1}
(\Phi_L \circ g\circ \gamma)'(t)=(f\circ\gamma )'(t)+L(\varphi\circ\gamma )'(t)=(\ud_xf+L)[(\varphi \circ \gamma)'(t)]
\end{equation}
for $\Mod_\infty$-a.e. curve fragment and a.e. $t \in {\rm dom}(\gamma)$. Thus, $\LIP[\varphi]\|\ud_xf+L\|$ is a $\ast$-upper gradient for $\Phi_L\circ g$, where $\|\cdot\|$ is the $\ell_2$ dual norm on $\R^n$. We get with \eqref{eq:ugphil}, the minimality of $|D(\Phi_L\circ g )|_*$ and the non-decreasing property of $\omega$
\begin{equation}\label{eq:ugphil2}
\Lip (\Phi_L\circ g)(x) \leq \omega(|D \Phi_L \circ g|_*(x))\leq \omega(\LIP[\varphi]\|\ud_xf+L\|)
\end{equation}
outside possibly another null set still labeled $N_L$. Now let $D\subset\mathcal B (\R^n,W)$ be a dense subset (thus $\{ \Phi_L:L\in D\}$ is dense in $S$). The right hand side of \eqref{eq:ugphil} restricted to $\Phi_L\in D$ are uniformly continuous outside a null set $B$ and let $N=\bigcup_{L\in D}N_L \cup B$. Then for every $x\in X\setminus N$ and every $L\in D$ we have \eqref{eq:ugphil2}; moreover, both sides in \eqref{eq:ugphil} are uniformly continuous over $D$ (thus continuous) and the identity \eqref{eq:ugphil1} is stable under limits of $L\in D$, allowing us to obtain \eqref{eq:ugphil2} for all $L\in \mathcal B(\R^n,W)$ by density.  Take 
 $L=-\ud_x f$ to get
 $\Lip(f-\ud_xf\circ\varphi)(x)=\Lip(\Phi_{\ud_x f}\circ g)=0$ $\mu$-a.e. $x$ and complete the proof.

\begin{remark}\label{rmk:ugscalarized}
In the spirit of \cite[Proposition 4.5]{tsd24}, one can obtain a representation of the $*$-upper gradient through a seminorm for separable classes of operators. This is similar to the proof of $(iii)\Rightarrow(i)$ of Theorem \ref{thm:RNP-LDS-char}. Our argument however avoids an abstract representation and thus we don't introduce a formal statement as in \cite{tsd24}.
\end{remark}

\section{Decomposability bundle}\label{sec:decomp-bundle}

Let $Y$ be a Banach space and $\mu$ a Radon measure on $Y$. A map $x\mapsto V(x)$, $Y\to \operatorname{Gr}(Y)$, is Borel if the set $\{(x,v) : x\in Y, v\in V(x)\}$ is Borel. Here $\operatorname{Gr}(Y)$ is the collection of all closed subspaces of $Y$. We restate Definition \ref{def:decomp-bundle} of a decomposability bundle with a bit more detail. A map $x\mapsto V(x)$ is called a candidate for the decomposability bundle if the following holds:
\begin{align}\label{eq:firstcond}
    \text{There exists a} &\text{  $\mu$-null set $N\subset X$ so that $\gamma'(t)\in V(\gamma_t)$} \\
    &\text{ for a.e. $t\in \dom(\gamma)$ for all $\gamma\in\Fr(Y)$ with $|\gamma\inv(N)|=0$.} \nonumber
\end{align}
The decomposability bundle $x\mapsto T_\mu(x)$ of $\mu$ is a Borel map $Y\to \operatorname{Gr}(Y)$ that satisfies \eqref{eq:firstcond} and if $x\mapsto V(x)$ is another map satisfying \eqref{eq:firstcond}, then $T_\mu(x)\subset V(x)$ $\mu$-a.e.  $x$. We will often refer to this as minimality.

The following properties can be obtained by standard arguments.

\begin{lemma}\label{lem:basic-prop-of-decomp-bundle}
The decomposability bundle of $\mu$ (if it exists) is $\mu$-a.e. unique: If $x\mapsto T_\mu(x)$ and $x\mapsto \widetilde T_\mu(x)$  are decomposability bundles as above, then $T_\mu=\widetilde T_\mu$ $\mu$-a.e. on $Y$. Moreover, if $\tilde\mu\ll\mu$ then 
\begin{align*}
    T_{\tilde\mu}(x)=T_\mu(x)\quad\tilde\mu\mbox{-a.e. }\; x.
\end{align*}
\end{lemma}
\begin{proof}
    The first claim follows from minimality in a direct way, so lets focus on the second one. From minimality, it is clear that $T_{\tilde{\mu}}(x)\subset T_\mu(x)$ for $\tilde{\mu}$-a.e. $x\in X$. since $T_\mu(x)$ satisfies \eqref{eq:firstcond} also with respect to $\tilde{
    \mu}$. 

    For $i\in \N$ we can find Borel sets $E_i$ s.t. $\mu|_{E_i} \ll \tilde{\mu}|_{E_i}$, where $\nu|_E(A)=\nu(E\cap A)$, and $\mu(X\setminus \bigcup E_i)=0$. Set $V(x)=T_{\tilde{\mu}}(x)$ for $x\in \bigcup_{i\in \N}E_i$ and otherwise $V(x)=T_\mu(x)$. We claim that $V$ satisfies \eqref{eq:firstcond} for $\mu$.

    Let $N_{\tilde{\mu}}$ be the null set for $T_{\tilde{\mu}}$ coming from \eqref{eq:firstcond}, and let $N$ be the one for $T_\mu$. Define $M=N\cup \bigcup_{i\in \N} N_{\tilde{\mu}}\cap E_i$. By construction $\mu(M)=0$. If $\gamma$ is a curve fragment, then for a.e. $t\in \dom(\gamma)$ for which $\gamma_t\not\in M$, we have either $\gamma_t\in E_i$ for some $i$ or not. For a full measure subet for which the second holds, then  $\gamma'(t)\in T_\mu(\gamma_t)$. For a full measure subset of $t$ such that $\gamma_t\in E_i$, we have $\gamma'(t)\in T_{\widetilde{\mu}}(\gamma_t)=V(x)$. The claim follows.
\end{proof}
 
The main results in this section are two existence results (Theorem \ref{thm:decomp-bundle} and Proposition \ref{prop:linear-part-of-blowup}) when $Y$ is an RNP-Banach space. The first guarantees existence of the decomposability bundle of any measure under an additional assumption on $Y$. The second shows existence under a finite dimensionality condition (on the measure, not the ambient space), and furthermore relates the decomposability bundle to the fragment-wise differential structure.

We say that a Borel regular measure $\mu$ on a Banach space $Y$ is concentrated on a separable set, if there exists a separable measurable set $E\subset Y$ s.t. $\mu(Y\setminus E)=0$. Note that if $\mu$ is a Borel regular measure s.t. $\mu(B(x,R))\in [0,\infty)$ for all $x\in Y, R>0$, then $\mu$ is easily seen to be concentrated on a separable set.
\begin{theorem}\label{thm:decomp-bundle}
Let Y be a Banach space so that $Y$ and $Y^*$ have the Radon-Nikodym property. Then any locally finite Borel regular measure, concentrated on a separable subset, admits a decomposability bundle.
\end{theorem}

In particular, any measure on $Y$ has a decomposability bundle whenever $Y$ is reflexive, or $Y^*$ is separable.

\begin{remark}\label{rmk:dunno}
We do not know whether the conclusion of Theorem \ref{thm:decomp-bundle} holds for the RNP space $Y=\ell^1$.
\end{remark}

In the second result we prove the existence of a decomposability bundle under a different assumption that $\mu$ is a measure on $X\subset Y$ where $Y$ is an RNP-Banach space and $(X,\mu)$ admits a fragment-wise chart. While the previous theorem applies to truly infinite dimensional settings, here the existence of the fragment-wise chart corresponds to a certain finite dimensionality of directions spanned by tangents of curve fragments.

\begin{proposition}\label{prop:linear-part-of-blowup} 
Let $(U,\varphi)$ be a $k$-dimensional fragment-wise chart of $(X,\mu)$. For $\mu$-a.e. $x\in U$ there exists a $k$-dimensional vector subspace $T_\mu(x)\subset Y$ and a linear bijection $L_{\varphi,x}:T_\mu(x)\to \R^k$ s.t.
\begin{itemize}
	\item[(i)] $x\mapsto T_\mu(x)$ is measurable as a map $U\to \operatorname{Gr}_k(Y)$;
	\item[(ii)] $\gamma_t'\in T_\mu({\gamma_t})$ a.e. $t\in \gamma\inv(U)$ for $\Mod_\infty$-a.e. $\gamma\in \Fr(X)$;
	\item[(iii)] $L_{\varphi,x}(\gamma_t')=(\varphi\circ\gamma)_t'$ a.e. $t\in \gamma\inv(U)$ for $\Mod_\infty$-a.e. $\gamma\in \Fr(X)$.
\end{itemize}
Moreover, $x\mapsto T_\mu(x)$ is the $\mu$-minimal bundle satisfying (i) and (ii). 
\end{proposition}

Our definition is different from the classical one in \cite{almar16}, which was originally given for measures in $\R^n$. We will show the equivalence in this case.

\begin{corollary}\label{cor:AM}
    If $\mu$ is a measure on $\R^n$, then the decomposability bundle $T_\mu$ agrees with the decomposability bundle of Alberti--Marchese \cite[Section 2.2]{almar16}.
\end{corollary}

\subsection{Existence in the finite dimensional case}

We define the cone in a direction of a subspace.
\begin{definition}\label{eq:conesubsspace} If $V$ is a $k$-dimensional subspace of a Banach space $Y$, $C(V,\delta)$ is the collection of all $k$-dimensional subspaces $W$ s.t. for every $w\in W$ there exists a $v\in V$ with $d(v,w)\leq \delta |w|$.
\end{definition}

\begin{proof}[Proof of Proposition \ref{prop:linear-part-of-blowup}]
Since $Y$ has RNP, given $\gamma\in\Fr(X)$, the derivative 
\[
\gamma_t'=\lim_{s\to t}\frac{\gamma_{s}-\gamma_t}{s-t}\in Y
\]
exists as a limit in $Y$ for a.e. $t\in \dom(\gamma)$. 
    
Let $\mathscr A_1,\ldots,\mathscr A_k$ be Alberti representations of $\mu|_U$ in the $\varphi$-direction of independent cones $C_1,\ldots, C_k$. Then, by a selection argument similar to \cite[Proposition 2.9]{bate12diff} there are measurable maps $\widetilde U\to\Fr(X):x\mapsto \gamma^x_l$, $l=1,\ldots,k$ defined on a full measure Borel set $\widetilde U\subset U$ so that the vector fields
\begin{align*}
&\Gamma_l:\widetilde U\to Y,\quad \Gamma_l(x)= (\gamma_l^x)'((\gamma_l^x)\inv(x))\\
&\Phi_l:\widetilde U\to \R^k,\quad \Phi_l(x)= (\varphi\circ \gamma_l^x)'((\gamma_l^x)\inv(x))
\end{align*}
are measurable, and 
\begin{itemize}
	\item $\Phi_l(x)\in C_l$ for $l=1,\ldots,k$;
	\item there is a constant $C>0$ so that $|Df|_\ast(x)\le C\|\nabla f(x)\|$ $\mu$-a.e. on $\widetilde U$ for any $f\in \LIP(X)$, where
	\[
	\nabla f(x)=\Big(\frac{(f\circ\gamma_1^x)'((\gamma_1^x)\inv(x))}{\|\Gamma_1(x)\|},\ldots,\frac{(f\circ\gamma_k^x)'((\gamma_k^x)\inv(x))}{\|\Gamma_k(x)\|}\Big).
	\]
\end{itemize}
Denote moreover $\widehat\Gamma_l(x):=\frac{\Gamma_l(x)}{\|\Gamma_l(x)\|}$ the normalized vector fields. Define
\begin{align}\label{eq:span}
T_\mu(x)=\operatorname{span}\{\Gamma_1(x),\ldots, \Gamma_k(x) \},\quad x\in \widetilde U.
\end{align}
Since the spanning vector fields are measurable, $x\mapsto T_\mu(x)$ is measurable, as a composition of a measurable map $x\mapsto (\Gamma_l(x))_{l=1}^k$ and a continuous map $(v_1,\dots, v_n)\mapsto \operatorname{span}\{v_1,\dots, v_n \}$, proving (i). Moreover, the linear map 
\begin{align}\label{eq:lin-map}
L_{\varphi,x}:T_\mu(x)\to \R^k,\quad \Gamma_l(x)\mapsto \Phi_l(x)
\end{align}
is well-defined and bijective.

We now prove (ii). To this end, fix $\varepsilon>0$. Notice that $W:=\overline{\operatorname{span}}\{X-X\}\subset Y$ is separable and $T_\mu(x)\subset W$ for all $x\in\widetilde U$. Let $\{V_i\}\subset \operatorname{Gr}_k(W)$ be a countable dense set. The sets
\begin{align*}
	U_i=\{x\in \widetilde U: \ T_\mu(x)\subset C(V_i,\varepsilon) \}
\end{align*}
cover $\widetilde U$.  Fix $i\in\N$ such that $\mu(U_i)>0$ and let $e_1,\ldots,e_k$ be a basis of $V_i$ such that $\|\widehat\Gamma_l(x)-e_l\|\le \varepsilon$ for each $x\in U_i$. Then for any linear functional $p\in Y^*$ with $V_i\subset \ker(p)$ we have that 
\begin{align*}
	\|\nabla p(x)\|=\|(p(\widehat\Gamma_1(x)),\ldots,p(\widehat\Gamma_k(x)))\|\le \sqrt k\|p\| \varepsilon
\end{align*}
for $\mu$-a.e. $x\in U_i$. Consequently $|Dp|_\ast\le C'\|p\|\varepsilon$ for some constant $C'>0$ independent of $\varepsilon>0$. It follows from this that 
\begin{align*}
	|(p\circ\gamma)_t'|=|p(\gamma_t')|\le C'\|p\|\varepsilon\|\gamma_t'\|\quad\textrm{a.e. }t\in \gamma\inv(U_i),\quad\Mod_\infty\textrm{-a.e. }\gamma\in \Fr(X).
\end{align*}
Since $W_i=\overline{\operatorname{span}}\{X-X,V_i\}\subset Y$ is separable (and contains $W$), as in the proof of Theorem \ref{thm:fragmentwise}, we can find a countable set $\{p_j\}\subset Y^*$ with norm one and $p_j|_{V_i}=0$ so that $\dist (w,V_i)=\sup_j|p_j(w)|$ for all $w\in W_i$. Thus 
\begin{align*}
\dist(\gamma_t',V_i)=\sup_j|p_j(\gamma_t')|\le C'\varepsilon\|\gamma_t'\|\quad\textrm{a.e. }t\in \gamma\inv(U_i)
\end{align*}
for $\Mod_\infty\textrm{-a.e. }\gamma\in \Fr(X).$ Consequently $\gamma_t'\in C(V_i,C'\varepsilon)\subset C(V_{\gamma(t)},(1+C')\varepsilon)$ a.e. $t\in \gamma\inv(U)$ for $\Mod_\infty$-a.e. $\gamma\in\Fr(X)$. Since $\varepsilon>0$ is arbitrary, this implies (ii).

To prove (iii), consider the graph $G(\varphi)=\{(x,\varphi(x)):\ x\in X\}\subset Y\times\R^k$ and the chart $\widehat U=\{(x,\varphi(x)):\ x\in U\}$, $\hat\varphi=p_{\R^k}|_{G(\varphi)}$. The vector fields corresponding to $\Gamma_l$ are now $(\Gamma_l,\Phi_l)$, and by the argument above we have that 
\begin{align*}
(\gamma_t',(\varphi\circ\gamma)_t')\in \operatorname{span}\{(\Gamma_1(\gamma_t),\Phi_1(\gamma_t)),\ldots, (\Gamma_k(\gamma_t),\Phi_k(\gamma_t)) \}=G(L_{\varphi,x})
\end{align*}
a.e. $t\in \gamma\inv(U)$ and $\Mod_\infty$-a.e. $\gamma\in\Fr(X)$.  

It remains to prove the minimality claim in the statement of Proposition \ref{prop:linear-part-of-blowup}. If $x\mapsto W_x$ is a bundle satisfying (ii), then in particular $\Gamma_1(x),\ldots,\Gamma_k(x)\in W_x$ $\mu$-a.e. $x\in U$. By \eqref{eq:span} this implies that $T_\mu(x)\subset W_x$ $\mu$-a.e. $x\in U$, proving minimality.
\end{proof}

 Next, we study the differentiability of $f\in \LIP(Y)$. We say that $f$ is Gateaux differentiable at $x$ in the direction of a subspace $V$, if there exists a linear map $\ud_Vf:V\to \R$ s.t. for every $v\in V$, we have
 \[
f(x+tv)=f(x)+\ud_{V}f(v)+o(|t|).
 \]
 We say that $f$ is Fr\'echet differentiable at $x$ in the direction of $V$, if there exists a linear map $\ud_Vf:V\to \R$ s.t. we have
 \[
f(x+v)=f(x)+\ud_{V}f(v)+o(\|v\|_Y).
 \]
 When $V$ is finite dimensional, Gateaux differentiability is easily seen to be equivalent to Fr\'echet differentiability.

Let $Y$ be an RNP Banach space. Maleva and Preiss have constructed a notion of a complete derivative for every $f\in \LIP(Y)$ \cite[Definition 4.1]{malevapreiss}. We give a slight weakening of this notion, that is easier to state and is sufficient for our purposes. A complete derivative assignment is a Borel measurable assignment to each $x\in Y$ of a subspace $V(x)$ in such a way that 
 \begin{enumerate}
     \item $f$ is Gateaux differentiable in the direction of $V(x)$ at $x$, and
     \item for every Lipschitz map $g:\R\to Y$ we have $g'(t)\in V(g(t))$ for a.e. $t\in \R$. 
 \end{enumerate}
 Recall that an assignment $x\mapsto V(x)$ is Borel measurable, if the set $\{(x,v) : x\in Y, v\in V(x)\} \subset Y\times Y\}$ is a Borel set. 
 In \cite{malevapreiss}, the second condition is replaced with a stronger condition for all Lipschitz mappings $g:X\to Y$ where $X$ is a Banach space, and an inclusion of directional derivatives $D_eg(x)$ in $V(g(x))$ is required for all $e\in Y$ and $x\in Y\setminus N$, whenever the directional derivative exists, for some \emph{null-set} $N$. When $X=\R$ the notion of null set is equivalent to being Lebesgue null, and we obtain our variant.
 
\begin{corollary}\label{cor:AM-diff}
Let $Y$ be a separable Banach space with the Radon-Nikodym property, and let $\mu$ be a locally finite Radon measure on $Y$. Assume that a decomposability bundle $T_\mu$ exists. Then, every $f\in \LIP(Y)$ is Gateaux differentiable in the direction of $T_\mu(x)$ at $\mu$-a.e. $x\in Y$. Moreover, if $T_\mu(x)$ is finite dimensional at $\mu$-a.e. $x$, then $f$ is Fr\'echet differentiable in the direction of $T_\mu(x)$ for $\mu$-a.e. $x\in Y$.
\end{corollary}

\begin{proof}
A complete derivative assignment $V(x)$ for $f$ exists by \cite[Proposition 5.1 and Theorem 5.7]{malevapreiss}. We claim that $T_\mu(x)\subset V(x)$ for $\mu$-a.e. $x$, from which the claim follows.

Let $\gamma \in \Fr(Y)$, and extend $\gamma$ by extrapolating linearly to give a map $g:\R\to Y$. By the definition of a complete derivative, for a.e. $t\in \R$ we have $g'(t)\in V(g(t))$. By Lebesgue differentiation a.e. $t\in \dom(\gamma)$ is not an isolated point. Thus, for a.e. $t\in \dom(\gamma)$ we also have $\gamma'(t)=g'(t)$ and $\gamma'(t)\in V(\gamma(t))$. This holds for all $\gamma\in \Fr(Y)$, and thus $x\mapsto V(x)$ is a candidate for the decompoability bundle of $\mu$. Therefore $T_\mu(x)\subset V(x)$ for $\mu$-a.e. $x$, and the claim follows.
\end{proof}
 
Before we prove Corollary \ref{cor:AM}, we briefly recall the exact definition from \cite[Section 2.6]{almar16} of a decomposability bundle. First, let $\mathcal{F}_\mu$ be the collection of all families $\{\mu_t : t\in I\}$, where $I$ is some measure space, with measure $dt$, and so that 
\begin{enumerate}
    \item $A\mapsto \mu_t(A)$ is Borel measurable and $\int \mu_t(\R^n) dt < \infty$,
    \item $\mu_t=\mathcal{H}^1|_E$ for some $1$-rectifiable $E_t$ for each $t\in I$,
    \item The measure $\nu$ given by $\nu(A)=\int \mu_t(A) dt$ is absolutely continuous with respect to $\mu$.
\end{enumerate}
Here, $\mathcal{H}^1$ is the usual Hausdorff $1-$measure.
The decomposablity bundle $V_\mu(x)$ is the $\mu$-a.e. minimal measurable map $V_\mu:\R^n \to {\rm Gr}(Y)$ s.t. for every family $\{\mu_t : t\in I\}$ in $\mathcal{F}_\mu$ it holds that ${\rm Tan}(E_t, x)\subset V_\mu(x)$ for a.e. $t\in I$ and $\mu_t$ a.e. $x\in E_t$, where ${\rm Tan}(E_t, x)$ is the weak tangent for a rectifiable set defined almost everywhere.

\begin{proof}[Proof of Corollary \ref{cor:AM}]
By the definition of $T_\mu(x)$, we have that there is a null set $N$ so that $\gamma'(t)\in T_{x}(\gamma_t)$ at a.e. $t\in \dom (\gamma )$ for all curve fragments $\gamma$ for which $|\gamma^{-1}(N)|=0$. Suppose that $E$ is a $1$-rectifiable set with $\mathcal{H}^1(E\cap N)=0$, then, by the definition of rectifiability, we can find curve fragments $\gamma_i:K_i\to E$ s.t. $\mathcal{H}^1(E \setminus \bigcup_{i\in \N} \gamma(K_i))=0$. Thus, $|\gamma_i^{-1}(N)|=0$, and therefore for a.e. $t\in K_i$, we have $\gamma'_i \in T_\mu(\gamma_i(t))$. For a.e. $t\in K_i$ we have ${\rm Tan}(E, \gamma_i(t))={\rm span}(\gamma'_i(t))$, and therefore for $\mathcal{H}^1$-a.e. $x\in E$ we have ${\rm Tan}(E, x)\subset T_\mu(x)$. Let now $\{\mu_t: t\in I\}\in \mathcal{F}_\mu$.  Since $N$ is a null set and $\nu(A)=\int \mu_t(A) dt$ is absolutely continuous with respect to $\mu$, we get 
\[
\int \mu_t(N) dt = \int \mathcal{H}^1(E_t \cap N) dt = 0.
\]
Thus, for a.e. $t\in I$ the set $E_t$ satisfies the conclusion that ${\rm Tan}(E, x)\subset T_\mu(x)$ for $\mu_t$-a.e. $x\in E_x$. By minimality of $V_\mu$, we get $V_\mu(x)\subset T_\mu(x)$ for $\mu$-a.e. $x\in \R^n$.

For the opposite inclusion, note that \cite[(ii) in Theorem 1.1]{almar16} implies that there exists a Lipschitz function $f\in \LIP(\R^n)$ that for  $\mu$-a.e. $x\in \R^n$ is not differentiable in any direction $v\not\in V_\mu(x)$.  Corollary \ref{cor:AM-diff} on the other hand shows that every $f\in \LIP(\R^n)$ is differentiable in the directions of $T_\mu(x)$ for $\mu$-a.e. $x\in \R^n$. Thus $T_{\mu}(x)\subset  V_\mu(x)$ for $\mu$-a.e. $x\in \R^n$.

\end{proof}

\subsection{Existence in the infinite dimensional case}

Throughout this subsection we fix an RNP-space $Y$ such that $Y^*$ also has RNP. The latter property is equivalent to $Y$ being Asplund, i.e. every separable subspace of $Y$ has separable dual \cite[Theorem 6]{asplund}. We refer to \cite{asplund} for an exposition of Asplund spaces and their many equivalent characterizations (including the ones mentioned above). See also \cite{asplundoriginal,namiokaphelps} for original treatises, where these spaces were introduced.
 
Let $Y_0 \subset Y$ be separable. Then $Y_0^*$ is separable. Fix a countable norm dense subset $\{e_i\} \subset B_{Y_0^*}$.
Set 
\[
\phi(V) = \sum_i 2^{-i} \| e_i|_V \|,\quad V\subset Y_0\;\mbox{ closed linear subspace}.
\]

Clearly $\phi(V) \le \phi(W)$ if $V \subset W$ and $0\le \phi(V)\le 1$ for every $V\subset Y_0$.

\begin{lemma}\label{lem:phi-distinguishes-subspaces}
If $V \subset W \subset Y_0$ are two subspaces, and $\phi(W) \le \phi(V)$, then $V = W$.
\end{lemma}

\begin{proof}
Since $\| e_i|_W \| \ge \| e_i|_V \|$ for all $i$, the condition
\[
\sum 2^{-i} \| e_i|_W \| = \phi(W) \le \phi(V) = \sum 2^{-i} \| e_i|_V \|
\]
implies that $\| e_i|_W \| = \| e_i|_V \|$ for all $i\in\N$.

Now suppose $V\subsetneq W$, and let $w \in B_W \setminus V$ and $e \in B_{Y_0^*}$ be such that
\[
e(w) = \|w\| > 0,\qquad e|_V = 0.
\]
By density of $\{e_i\}\subset B_{Y^*_0}$ there exists a sequence $(e_{i_k})_{k\in \N}$ s.t. $e_{i_k} \to_{k\to\infty} e$ in norm. Thus, $\|w\|=e(w) = \lim e_{i_k}(w)$ and $\| e_{i_k}|_V \| < \|w\|/4$ for large $k$. This implies that 
\[
\|e_{i_k}|_V\|\le \|w\|/4<3\|w\|/4\le e_{i_k}(w)\le \|e_{i_k}|_W\|
\]
for large $k$, contradicting the fact that $\| e_i|_W \| = \| e_i|_V \|$ for all $i\in\N$. Thus $V = W$.
\end{proof}

\begin{proof}[Proof of Theorem \ref{thm:decomp-bundle}]
Let $\mu$ be a boundedly finite Borel regular measure on $Y$  that is concentrated on a separable set. We can assume that $\mu$ is finite, cf. Lemma \ref{lem:basic-prop-of-decomp-bundle}. Then $\spt\mu$, and thus $Y_0 = \overline{\mbox{span}}\{\spt\mu\}$, is separable.

Consider the collection $\mathcal C$ of Borel bundles $x \mapsto V_x \in \operatorname{Gr}(Y_0)$ with the following property:
\begin{align}\label{eq:tang-bundle}
\gamma_t'\in V_{\gamma(t)}\quad \mbox{a.e.}\; t\in \gamma\inv(\spt\mu)    
\end{align}
for $\Mod_\infty^\mu$-a.e. $\gamma\in \Fr(\spt\mu)$. The collection $\mathcal C$ is non-empty because the constant bundle $x\mapsto Y_0$ belongs to $\mathcal C$. We moreover remark that if $(V^i)\subset \mathcal C$ is a countable set, then the pointwise (countable) intersection $\displaystyle \bigcap_iV^i$ belongs to $\mathcal C$. Indeed, if $\Gamma_i$ are the $\Mod_\infty^\mu$-null families of fragments where \eqref{eq:tang-bundle} fails for $V^i$, then $\Gamma:=\bigcup_i\Gamma_i$ is $\Mod_\infty^\mu$-null and for each $\gamma\notin\Gamma$ there are null-sets $N_i\subset \dom(\gamma)$ so that we have $\gamma_t'\in V_{\gamma(t)}^i$ for $t\in \gamma\inv(\spt\mu)\setminus N_i$.  Thus $\gamma_t'\in \bigcap_iV_{\gamma(t)}^i$ for all $t\in \dom(\gamma)\setminus N$, where $N:=\bigcup_iN_i\subset \dom(\gamma)$. Since $N$ is a null-set, $\bigcap_i V^i$ satisfies \eqref{eq:tang-bundle} and is thus in $\mathcal{C}$.

 Note that, for each $V \in \mathcal C$, and $e\in Y_0^*$ the map $\phi_e:x\mapsto \|e|_{V(x)}\|$ is $\mu-$measurable. Indeed, $\phi_e^{-1}(a,\infty)=\pi(\{(x,v): e(v)\in (a,b), |v|\leq 1, v\in V(x)\})$ is a Suslin set, since $\{(x,v): e(v)\in (a,b), |v|\leq 1, v\in V(x)\}$ is Borel. Here, $\pi:Y\times Y \to Y$ is the projection map $\pi(x,v)=x$. As a sum of such maps $x \mapsto \phi(V_x)$ is $\mu$-measurable.m Consider the minimization problem
\[
\inf_{V \in \mathcal C} \int \phi(V_x)\, d\mu(x) =: A,
\]
and let $(V^n) \subset \mathcal C$ be a minimizing sequence. The intersection $W_x = \bigcap V^n_x$ belongs to $\mathcal C$ as discussed above. It follows
that
\[
A = \int \phi(W_x)\, d\mu = \lim_{n\to\infty} \int \phi(V^n_x)\, d\mu(x) = A,
\]
i.e. $W \in C$ minimizes
\[
\int \phi(V_x)\, d\mu
\]
among $V \in \mathcal C$. We claim that $W$ is $\mu$-a.e. minimal with respect to inclusion, i.e. if $V \in \mathcal C$ then $W_x \subset V_x$ $\mu$–a.e.

Indeed, given $V \in \mathcal C$, the intersection $V \cap W$ belongs to $\mathcal C$. The pointwise inequality
\[
\phi(V_x \cap W_x) \le  \phi(W_x) \quad \mu\text{–a.e.}
\]
and the integral inequality
\[
\int \phi(W_x)\, d\mu \le \int \phi(V_x \cap W_x)\, d\mu
\]
together imply $\phi(W_x) = \phi(V_x \cap W_x)$ $\mu$–a.e., which by Lemma \ref{lem:phi-distinguishes-subspaces} implies that $W = V \cap W$ $\mu$–a.e., completing the proof of the claim. 

We have proven that $W$ is the $\mu$-a.e. minimal Borel bundle satisfying \eqref{eq:tang-bundle}, thus it is the decomposability bundle of $\mu$.
\end{proof}

\subsection{Density of directions} 

For measures on Banach spaces finite dimensional fragment-wise charts can be taken as linear maps, as follows from the following lemma.  

\begin{lemma}\label{lem:proj-are-charts}
Suppose that $Y$ is an RNP Banach space and  $\mu$ a Radon measure on $Y$ that is supported on $X$ has a decomposability bundle $T_\mu$. Let $p:Y\to V$ be a bounded linear map to an $n$ dimensional normed space $V$ and $U\subset X$ such that $\mu(U)>0$ and $p|_{T_\mu(x)}$ is bijective for $\mu$-a.e. $x\in U$. Then $(U,p)$ is an $n$-dimensional fragment-wise chart.
\end{lemma}
\begin{proof}
The statement is independent of the choice of norm on $V$, and thus we can assume $V=\R^n$ with the standard norm. To prove that $(U,p)$ is a fragment-wise chart we will check the equivalent condition of differentiability in Lemma \ref{lem:chartchar}. Let $f:X\to \R^n$ be Lipschitz. By McShane extension theorem, we may extend it as a Lipshitz function $f:Y\to \R$. By Corollary \ref{cor:AM-diff}, $f$ is differentiable $\mu$-a.e. with respect to $T_\mu(x)$ with derivative $\ud_{T_{\mu}(x)}f$. Define $\ud_x f(v)=\ud_{T_{\mu}(x)}f(p|_{T_{\mu}(x)}^{-1}(v))$. For $\Mod_\infty$-a.e. $\gamma$ and a.e. $t\in \dom(\gamma)$ we have 
\begin{enumerate}
    \item $\gamma'_t\in T_\mu(\gamma_t)$ and
    \item $(f\circ \gamma)'_t = \ud_{\gamma_t}f$.
\end{enumerate}
From these two conditions, it follows that $\ud_x f$ is a fragment-wise differential with respect to $p$.

Next, we argue uniqueness. Suppose $\xi_1,\xi_2$ were two fragment-wise differentials for some $f$. Let
\[
T^\xi_\mu(x)=\{v\in T_\mu(x) : \xi_1(p(v))=\xi_2(p(v))\}.
\]
For $\Mod_\infty$-a.e. $\gamma\in \Fr(X)$ and a.e. $t\in \dom(\gamma)$ we have $(f\circ \gamma)_t' = \xi_i(p(\gamma_t'))$ for both $i=1,2$. Thus, we must have $\gamma_t'\in T^\xi_\mu(\gamma_t)$ for $\Mod_\infty$-a.e. $\gamma\in \Fr(X)$ and a.e. $t\in \dom(\gamma)$. Thus, $T^\xi_\mu(\gamma_t)$ is a candidate for the decomposability bundle, and by minimality of $T_\mu$ we have $T^\xi_\mu(\gamma_t)=T_\mu(\gamma_t)$ for $\mu$-a.e. $x\in X$. This implies that $\xi_1=\xi_2$ for $\mu$-a.e. $x\in U$, since $p$ is a bijective on $T_\mu(x)$.
\end{proof}

We end this subsection with a version of density of directions, which was established in a slightly different form in \cite[Theorem 6.4]{tsd24}.

\begin{lemma}\label{lem:densityofdir} Assume that $\mu$ is a Radon measure supported on a subset $X$ of an RNP Banach space $Y$ that has a finite dimensional decomposability bundle $T_\mu(x)$. Then, for $\mu$-a.e. $x$, the set 
\[
\{\gamma'_t \in Y:\ t \text{ is a Lebesgue point of } \dom(\gamma), \gamma_t=x, \text{ and } \gamma'_t \text{ exists }\}
\]
is dense in $T_\mu(x)$.
\end{lemma}
\begin{proof}
    We can cover $\mu$-a.e. of the space by measurable sets $U$ of positive measure, that have linear projections $p:T_\mu(x)\to \R^n$ that is bijective for $\mu$-a.e. $x\in U$. We argue for a.e. $x\in U$, and then the full claim follows by a covering argument.

    By Lemma \ref{lem:proj-are-charts}, $(U,p)$ is a fragment-wise chart for $(X,\mu)$. Let $C$ be an arbitrary cone in $\R^n$. We will show that for $\mu$-a.e. $x\in U$ there exists $\gamma\in \Fr(X)$ and $t\in \dom(\gamma)$ with $\gamma_t=x$, and $\gamma'_t \in p^{-1}(C)\cap T_\mu(x)$. From this, varying $C$ and a linear reparametrization of $\gamma$, density follows. Let $D\subset U$ be the set of points where there does not exist such a curve fragment. The complement of $D$ is a continuous image of the evaluation map from $\Fr(X)\times \R$ of a Borel set, and thus a Suslin set and $\mu$-measurable; cf. \cite[Lemma A.1 and Lemma A.2]{ts24} for similar arguments. Thus, $D$ is also $\mu$-measurable.

    Since $p$ is a fragment-wise chart, the measure $\mu|_U$ admits $n$ $p$-independent Alberti representations. By \cite[Corollary 3.95]{sch16b} $\mu|_U$ also admits an Alberti representation $\{\mu_\gamma, \mathbb{P}\}$ in $p$-direction of $C$.

    Let now $N$ be the set given by $T_\mu$ being a decomposability bundle: for all $\gamma\in \Fr(X)$ and a.e. $t\in \dom(\gamma)$ with $\gamma_t\not\in N$ we have $\gamma'_t\in T_\mu(\gamma_t)$.  We have
    \[
    0=\mu(N\cap U)=\int_{\Fr(X)} \mu_\gamma(N) d\mathbb{P},
    \]
    and thus $\mu_\gamma(N\cap U)=0$ for $\mathbb{P}$-a.e. $\gamma$. Thus, for $\mu_\gamma$-a.e. $t\in \dom(\gamma)$ we have $\gamma'_t\in T_\mu(\gamma_t)$. Since also $(p\circ \gamma)'_t\in C$ for a.e. $t\in \dom(\gamma)$ and $\mathbb{P}$-a.e. $\gamma\in {\Fr(X)}$, we have that $\mu_\gamma(D)=0$ for $\mathbb{P}$-a.e. $\gamma\in {\Fr(X)}$. Again by the definition of an Alberti representation,
    \[
    0=\int_{\Fr(X)} \mu_\gamma(D) d\mathbb{P}=\mu(D).
    \]
    Thus, $\mu(D)=0$, and the claim follows.
\end{proof}

\section{Subsets of RNP Banach spaces}\label{sec:RNP}

Throughout this section we suppose $(X,\mu)$ is a metric measure space where $\mu$ vanishes on porous sets. Suppose further that $X\subset B$ isometrically for some Banach space $B$, and fix a non-principal ultrafilter $\omega$ on $\N$. Recall that the \emph{ultralimit} of a sequence $(Y_i,\bar y_i)$ of metric spaces is a pointed metric space $(Y_\omega,y_\omega)$ consisting of equivalence classes $[y_i]_\omega$ of sequences $(y_i)$ with $y_i\in Y_i$ and $\sup_id(y_i,\bar y_i)<\infty$, where $(y_i)\sim(z_i)$ if $\lim_\omega d(y_i,z_i)=0$. Here $\lim_\omega$ is the ultralimit of the sequence of real numbers along the principal ultrafilter $\omega$. See e.g. \cite{pasqualettoschultz} for discussion on this terminology and some references. The quantity
\begin{align*}
    d_\omega([y_i]_\omega,[z_i]_\omega)=\lim_\omega d(y_i,z_i)
\end{align*}
defines a metric on $Y_\omega$, and the basepoint is $y_\omega=[\bar y_i]_\omega$. If $B$ is a Banach space and $(Y_i,\bar y_i)=(B,0)$ for all $i\in\N$ is a constant sequence, the ultralimit is called the \emph{ultrapower} and denoted $B^\omega$. $B^\omega$ is a Banach space and there is a canonical isometric  embedding 
\[
B\hookrightarrow B^\omega,\quad v\mapsto [v,v,\ldots]_\omega.
\]
The embedding is surjective, i.e. $B=B^\omega$, if and only if $B$ is finite dimensional. Moreover, a bounded linear map $L:B\to V$ between Banach space admits an ultrapower $L^\omega:B^\omega\to V^\omega$, which is a bounded linear map given by 
\begin{align*}
L^\omega([v_i]_\omega):=[L v_i]_\omega    
\end{align*}
(If $V$ is finite dimensional, we regard $L^\omega$ as a map $B^\omega\to V$, since $V^\omega$ is isomorphic to $V$, with the isomorphism given by $[v_i]_\omega \to \lim_\omega v_i$, where this limit can be evaluated componentwise.) We often implicitly identify (subsets of) $B$ with it's image in $B^\omega$ under the canonical embedding. In particular we may regard $T_\mu(x)\subset B\hookrightarrow B^\omega$ as a subspace of $B^\omega$.

For $\mu$-a.e. $x\in X$ any $(Y,\nu,o)\in \Tan(X,\mu,x)$ is doubling (and thus proper), see Proposition \ref{prop:porous-tangent}. If the convergence $T_{x,r_j}X\stackrel{pmGH}{\longrightarrow}Y$ is realised by a sequence of scales $r_j\to 0$, then $(Y,o)$ is isometric to the ultralimit of the sequence $(r_j\inv X,x)$, see e.g. \cite[Corollary 9.3]{pasqualettoschultz}. Thus the isometric embeddings
\[
\iota_j:(r_j\inv X,x)\to B\subset B^\omega,\quad y\mapsto \frac{y-x}{r_j}
\]
have an $\omega$-limit 
\[
\iota_Y:(Y,o)\to B^\omega,\quad y\mapsto \Big[\frac{y-x}{r_j}\Big]_\omega
\]
which is an isometric embedding. Thus we may regard any tangent $(Y,\nu,o)\in\Tan(X,\mu,x)$ as a subset of $B^\omega$ for $\mu$-a.e. $x\in X$. We fix this setup and use the observation above in the following subsections without further mention.

The aim of this section is to study tangents of subspaces of an RNP Banach space and how they partially inherit the linear structure. We present now the main result of this section.
\begin{theorem}\label{thm:tang=prod}
Let $B$ be a Banach space with the RNP and $(X,\mu )$ a metric measure space in $B$ such that $\mu$ vanishes on porous sets. Suppose $(U,\varphi)$ is an $n$-dimensional fragment-wise chart of $(X,\mu)$. Then for $\mu$-a.e. $x\in U$ the following holds: For any $(Y,\nu,o)\in \Tan(X,\mu,x)$ and any projection $p:B^\omega\to T_\mu(x)$, there exists a set $Z\subset \ker(p)$ and a doubling measure $\nu_Z$ on $Z$ so that 
\begin{itemize}
    \item[(i)] $(Y,o)=(Z+T_\mu(x),0)$ as pointed metric spaces;
    \item[(ii)] $\nu=\nu_Z\times \leb^n$.

\end{itemize}
\end{theorem}

\subsection{Infinitesimal geometry and the decomposability bundle}

In an RNP Banach space, the decomposability bundle must be asymtotically close to the set. This is shown in the following Lemma. In what follows, for $R>0$ we denote $B(R):=\{ v\in B:\|v\| < R\}$ and $S(R):=\{v\in B : \|v\|=R\}$ the ball and sphere of radius $R$ in $B$, respectively. 

\begin{lemma}\label{lem:decomp-bundle-subset-of-tangent}
For $\mu$-a.e. $x\in X$ the following holds: for any $\varepsilon,R>0$ there exists $r_0=r(x,\varepsilon,R)>0$ so that 
\begin{align*}
\sup_{v\in T_\mu(x)\cap B(R)}\dist\Big(v,\frac{X-x}{r}\Big)\le \varepsilon
\end{align*}
whenever $r\in (0,r_0)$.
\end{lemma}
\begin{proof}
 By the density of directions (Lemma \ref{lem:densityofdir}) and linearly reparametrizing, the set $$ \{\gamma_0': \gamma_0=x,\ 0\mbox{ Lebesgue density point of }\dom(\gamma),\ |\gamma_0'|=1\}$$ is dense in $S(1)\cap T_\mu(x)$ for $\mu$-a.e. $x\in X$. It follows that 
$$ \{\gamma_0': \gamma_0=x,\ 0\mbox{ Lebesgue density point of }\dom(\gamma),\ |\gamma_0'|\le R\}$$ 
is dense in $T_\mu(x)\cap B(R)$. Let $\delta=\varepsilon/4$ and let $\{\gamma_1'(0),\ldots,\gamma_M'(0)\}$ be a $\delta$-net in $T_\mu(x)\cap B(R)$. Denote $L=\max_{l=1,\ldots,M}\LIP(\gamma_l)$. There exists $r_0=r_0(x,\varepsilon,R)$ such that 
\begin{align*}
    &\Big\|\gamma_l'(0)-\frac{\gamma_l(r)-x}{r}\Big\|<\delta,\quad r\in \dom(\gamma_l)\cap (0,r_0),\quad l=1,\ldots,M\\
    &\frac{\leb^1([-r,r]\setminus \dom(\gamma_l))}{r}<\delta/L,\quad r\in (0,r_0), ,\quad l=1,\ldots,M.
\end{align*}
Now let $v\in T_\mu(x)\cap B(R)$ and $r\in (0,r_0)$. By the conditions above we find $l\in \{1,\ldots,M\}$ with $\|v-\gamma_l'(0)\|<2\delta$, and $\tilde r\in \dom(\gamma_l)$ with $|\tilde r-r|<\delta r/L$. Note that
\begin{align*}
\Big\|\frac{\gamma_l(\tilde r)-x}{\tilde r}-\frac{\gamma_l(\tilde r)-x}{r}\Big\|=\|\gamma_l(\tilde r)-x\|\frac{|r-\tilde r|}{\tilde r r}\le L\tilde r\cdot \frac{\delta}{\tilde r L}=\delta
\end{align*}

Thus
\begin{align*}
\dist(v,r\inv(X-x))\le \Big\|v-\frac{\gamma_l(\tilde r)-x}{r}\Big\|\le &\|v-\gamma_l'(0)\|+\Big\|\gamma_l'(0)-\frac{\gamma_l(\tilde r)-x}{\tilde r}\Big\| \\
&+\Big\|\frac{\gamma_l(\tilde r)-x}{\tilde r}-\frac{\gamma_l(\tilde r)-x}{r}\Big\|\\
\le &2\delta+\delta+\delta=\varepsilon,
\end{align*}
completing the proof.
\end{proof}

The following lemma shows that the rescalings $r\inv(X-x)$ of $X$ are close to a (non-Euclidean) product whose one factor is the decomposability bundle in the Hausdorff sense. This will be useful in the proof of Theorems \ref{thm:le donne} and \ref{thm:tang=prod}, as well as Proposition \ref{prop:const z-component}. This result is similar to \cite[Proposition 2.9]{delninmerlo} and \cite[Theorem 1.2]{davideb20}.

\begin{lemma}\label{lem:e-close-to-product}
Let $(X,\mu)\subset B$ be a metric measure space where $\mu$ vanishes on porous sets. Then for $\mu$-a.e. $x$ the following holds: given a projection $p_x:B\to T_\mu(x)$, there are sets $Z_\rho\subset \ker p_x$, $\rho>0$, so that 
\begin{align}
&\sup\{\dist(y,r\inv Z_{rR}+T_\mu(x)):\ y\in r\inv(X-x)\cap B(R)\}\stackrel{r\to 0}{\longrightarrow} 0\label{eq:e1}\\
&\sup\{\dist(z,r\inv(X-x)):\ z\in (r\inv Z_{rR}+T_\mu(x))\cap B(R)\}\stackrel{r\to 0}{\longrightarrow} 0\label{eq:e2}
\end{align}
for all $R>0$.
\end{lemma}
In the proof of this Lemma as well as the next Theorem we  denote by $\alpha(V,W)$ the Hausdorff distance of the sets $V\cap B(1)$ and $W\cap B(1)$: 
\[
\alpha(V,W)=\sup_{w\in W\cap B(1)}\inf_{v\in V} d(v,w) + \sup_{v\in V\cap B(1)}\inf_{w\in W} d(v,w).
\]
\begin{proof}
Define $Z_\rho=(\id-p_x)(B_X(x,\rho)-x)$. Then $r\inv(X-x)\cap B(R)\subset r\inv Z_{rR}+T_\mu (x)$. Thus 
\begin{align*}
    e_{B(R)}(r\inv(X-x),r\inv(Z_{rR}+T_\mu (x)))=0,
\end{align*}
proving \eqref{eq:e1}. Here, recall the definition of the excess $e_{B(R)}$ given in \eqref{eq:excess}.

To prove \eqref{eq:e2} let, for each $m\in \N$, $\{K_j^m\}_{j\in \N}$ be a partition of $X$ (up to a $\mu$-null set) into compact sets so that for each $j$:
\begin{itemize}
    \item[(1)] $\alpha(T_\mu(x),T_\mu(y))\le m^{-2}$ for all $x,y\in K_j^m$;
    \item[(2)] the convergence $e_{B(m)}(T_\mu(x),r\inv(X-x))\stackrel{r\to 0}{\longrightarrow} 0$ is uniform in $K^m_j$, cf. Lemma \ref{lem:decomp-bundle-subset-of-tangent}.
\end{itemize}

Observe that $\mu$-a.e. $x\in U$ has the property that, for every $m\in\N$ there exists $j_m$ so that $x$ is a Lebesque density point and a point of non-porosity of $K_j^m$ for $j=j_m$. We fix such $x$, and let $\varepsilon>0$ be arbitrary. For any $m$ and $y\in K^m_{j_m}$ we have $T_\mu(x)\cap B(m)\subset C(T_\mu(y),m^{-2})$ by (1), while by (2) there exists $r_0$ so that $T_\mu(y)\cap B(m)\subset B(r\inv(X-y),\varepsilon)$ for $r\in (0,r_0)$. Together these inclusions imply that $T_\mu(x)\cap B(m)\subset B(r\inv(X-y),\varepsilon+m\inv)$ for $r\in (0,r_0)$, which can be rewritten as
\begin{align}\label{eq:e3}
    \frac{y-x}{r}+T_\mu(x)\cap B(m)\subset B(r\inv(X-x),\varepsilon+m\inv),\quad r\in (0,r_0).
\end{align}
Since $K_{j_m}^m$ is $\varepsilon r$-dense in $X\cap B(x,rm)$ for small enough $r$, there exists $r_1\le r_0$ so that \eqref{eq:e3} holds with $2\varepsilon+m\inv$ in place of $\varepsilon+m\inv$ for all $y\in X \cap B(x,rm)$, $r\in (0,r_1)$. It follows that 
\begin{align*}
    (r\inv Z_{rm}+T_\mu(x))\cap B\Big(\frac{m}{1+\|p_x\|}\Big)\subset B(r\inv(X-x),2\varepsilon+m\inv).
\end{align*}
Since $\varepsilon>0$ and $m$ are arbitrary \eqref{eq:e2} follows.
\end{proof}

We can now combine these lemmas to obtain the product structure for tangents. Here, it is also crucial that we obtain the product sructure for the blown up measure $\nu$. Recall the notation $\iota_Y:(Y,o)\to B^\omega$, $y\mapsto \Big[\frac{y-x}{r_j}\Big]_\omega$ and $p^\omega:B^\omega\to V$ from the beginning of the section.

\begin{proof}[Proof of Theorem \ref{thm:tang=prod}]
To prove (i) we argue as in Lemma \ref{lem:e-close-to-product}. For each $m$, let $\{K_j^m\}$ be a partition of $U$ (up to a $\mu$-null set) into compact sets so that for each $j$:
\begin{itemize}
    \item[(1)] $\alpha(T_\mu(x),T_\mu(y))\le m^{-2}$ for all $x,y\in K_j^m$;
    \item[(2)] the convergence $e_{B(m)}(T_\mu(x),r\inv(X-x))\stackrel{r\to 0}{\longrightarrow} 0$ is uniform in $K^m_j$, cf. Lemma \ref{lem:decomp-bundle-subset-of-tangent}.
\end{itemize}

For $\mu$-a.e. $x\in U$ and every $m\in\N$ there exists $j_m$ so that $x$ is a Lebesque density point and a point of non-porosity of $K_j^m$ for $j=j_m$. We fix such $x$. If $(Y,\nu,o)\in\Tan(X,\mu,x)$ is realised by a sequence of scales $r_i\downarrow 0$, then for any $y\in Y$ there exists a sequence $(y_i)\subset K^m_{j_m}$ representing $y$; in particular $\iota_Y(y)=[r_i\inv(y_i-x)]_\omega$. By (2), for any $\varepsilon>0$ there exists $r_0>0$ so that
\begin{align*}
T_\mu({y_i})\cap B(m)\subset B(r\inv(X-y_i),\varepsilon),\quad r<r_0.
\end{align*}
By (1) we have $T_\mu(x)\cap B(m)\subset C(T_\mu({y_i}),1/m^2)$, and together these inclusions imply 
\begin{align*}
    T_\mu(x)\cap B(m)\subset B(r\inv(X-y_i),\varepsilon+m^{-1}), r<r_0.
\end{align*}
Thus, for large $i$ we have $\frac{y_i-x}{r_i}+T_\mu(x)\cap B(m)\subset B(r_i\inv(X-x),\varepsilon+m\inv)$. Since $m$ and $\varepsilon$ are arbitrary, this implies 
\begin{align}\label{eq:V_x-invariance}
    \iota_Y(y)+T_\mu(x)\subset \iota_Y(Y)\quad \mbox{ for any }\;y\in Y. 
\end{align}
Now let $p:B^\omega\to T_\mu(x)$ be any linear projection, and set $Z:=(\id-p)(\iota_Y(Y))$. Since $\iota_Y(y)=(\id-p)(\iota_Y(y))+p(\iota_Y(y))\in Z+T_\mu(x)$ for every $y\in Y$ we have $\iota_Y(Y)\subset T_\mu(x)+Z$. Conversely, if $\bar y\in T_\mu(x)+Z$, there is $v\in T_\mu(x)$ and $y\in Y$ such that $\bar y=v+(\id-p)(\iota_Y(y))$. Since $v-p(\iota_Y(y))\in T_\mu(x)$, it follows from \eqref{eq:V_x-invariance} that $\bar y =\iota_Y(y)+v-p(\iota_Y(y))\in \iota_Y(Y)$,  proving $T_\mu(x)+Z\subset \iota_Y(Y)$. Since $\iota_Y:Y\to \iota_Y(Y)$ is an isometric homeomorphism, this proves (i).

Let $x$ be a density point of a compact set $K\subset U$ so that the decomposability bundle of $\mu$ restricted to $K$ is continuous. Then, for any projection $p_x:B\to T_\mu(x)$ there exists $r_0$ so that $K_x:=K\cap \bar B(x,r_0)$ has almost full $\mu$-measure in $B(x,r_0)$ and $(K_x,p_x)$ is a fragment-wise chart (of dimension $n$), cf. Lemma \ref{lem:proj-are-charts}. Let $(Z+T_\mu(x),\nu,z_0)\in \Tan(X,\mu,x)$ where $Z\subset \ker(p^\omega_x)$. Then it is direct to show by identifying an ultralimit with the pmGH-tangent (cf. \cite[Corollary 9.3]{pasqualettoschultz}) that the blow-up of $p_x$ at $x$ is the restriction of the ultralimit $p_x^\omega:B^\omega\to T_\mu(x)$ to $Z+T_\mu(x)$, i.e. $\psi=p_x^\omega|_{Z+T_\mu(x)}:Z+T_\mu(x)\to T_\mu(x)$; thus 
\begin{align*}
    \psi(z+v)=v,\quad z+v\in Z+T_\mu(x).
\end{align*}
Fix a basis $\{v_1,\ldots,v_n\}$ of $T_\mu(x)$ consisting of unit vectors. By \cite[Corollary 3.14]{sch16a}, for each $v_i$, $i=1,\ldots,n$, there exists a measure $\P_i$ on the set of geodesic lines $\gamma: \R\to Z+T_\mu(x)$ with $(\psi\circ\gamma)'\equiv v_i$ so that 
\begin{align}\label{eq:AR}
\nu(A)=\int \Ha^1(\im(\gamma)\cap A)\ud\P_i(\gamma),\quad A\subset Z+T_\mu(x)\; \mbox{ Borel.}
\end{align}
Denoting $W_i=\operatorname{span}\{v_l: l\ne i\}$ and post-composing the geodesic lines $\gamma$ with suitable translations we may assume that $\P_i$ is concentrated on the space 
\begin{align*}
G_i:=\{\gamma:\R\to Z+T_\mu(x)\mbox{ geodesic line with }(\psi\circ\gamma)'\equiv v_i,\ \psi(\gamma_0)\in W_i\}
\end{align*}
For each $\gamma=(\gamma_Z,\gamma_x)\in G_i$, we have $\gamma_x(t)=w_\gamma+v_it$. Moreover, by Proposition \ref{prop:const z-component} we have that $\gamma_Z$ is constant for $\P_i$-a.e. $\gamma$.

Fix a bounded Borel set $B\subset Z$ and consider the measure $\nu_B:=p_{x\ast}(\nu|_{B\times T_\mu(x)})$ on $T_\mu(x)$. We will identify $T_\mu(x)$ with $\R^n$ through the basis $\{v_1,\ldots,v_n\}$, and write $E\times [a,b]$ for $E+v_i[a,b]$ for a given Borel set $E\subset W_i$ and $a<b$. Using \eqref{eq:AR} we have
\begin{align*}
\nu_B(E\times[a,b])&=\int \Ha^1(B+(E\times [a,b])\cap\im(\gamma))\ud\P_i\\
&=(b-a)\int\chi_{e_0\inv(B+E)}\ud\P_i=(b-a)e_{0\ast}\P_i(B+E).
\end{align*}
It follows that $\nu_B$ is a product measure $\nu_B=\leb^1\times\sigma_i^B$ where $\sigma_i^B$ is the measure on $W_i$ defined by $\sigma_i^B(E)=e_{0\ast}\P_i(B+E)$.

For distinct $i,j\in \{1,\ldots,n\}$, let $W_{ij}=W_i\cap W_j$. For $E\subset W_{ij}$ and $A_i\subset \operatorname{span}\{v_i\}$, $A_j\subset \operatorname{span}\{v_j\}$ we then have 
\begin{align*}
\nu_B(E\times A_i\times A_j)=\sigma_i^B(E\times A_j)\leb^1(A_i)=\sigma_j^B(E\times A_i)\leb^1(A_j).
\end{align*}
It follows that 
\begin{align*}
    \sigma_{ij}^B(E):=\frac{\nu_B(E\times A_i\times A_j)}{\leb^1(A_i)\leb^1(A_j)}, \quad E\subset W_{ij}
\end{align*}
is independent of the choice of $A_i,A_j$ and defines a measure on $W_{ij}$ satisfying $\nu_B=\sigma_{ij}^B\times \leb^1\times \leb^1$. Repeating this argument we obtain
\begin{align*}
    \nu_B=c_B\leb^1\times\cdots\times\leb^1=c_B\leb^n
\end{align*}
for each bounded Borel $B\subset Z$. Since $\nu_B$ and $\leb^n$ are measures it follows that $B\mapsto c_B$ can be extended as a measure on $Z$, which we denote $\nu_Z$. It satisfies 
\begin{align*}
    \nu(B\times A)=\nu_B(A)=\nu_Z(B)\leb^n(A)
\end{align*}
for all Borel $A\subset T_\mu(x)$, $B\subset Z$, so that $\nu=\nu_Z\times \leb^n$. The fact that $\nu_Z$ is doubling follows from the fact that $\nu$ and $\leb^n$ are doubling measures.
\end{proof}

\section{Rectifiability of LDS subsets of RNP-spaces}\label{sec:extra-factors}

In this section we prove Theorems \ref{thm:LDS-in-RNP} and \ref{thm:le donne}. The technically most involved result in our argument is the following theorem excluding non-trivial extra factors in generic blow-ups, cf. Theorem \ref{thm:tang=prod}. It may be regarded the main result of this section. 

\begin{theorem}\label{thm:no-extra-factors} Let $(X,\mu)$ be an LDS contained in an RNP-Banach space $B$. 
	Let $(U,\varphi)$ be a Cheeger chart of dimension $n$ in $(X,\mu)$. Then for $\mu$-a.e. $x\in U$, every tangent $(Y,\nu,o)\in \Tan(X,\mu,x)$ is of the form $(Y,\nu,o)=(T_\mu(x),c\leb^n,0)$.
\end{theorem}

Together with the following proposition, Theorem \ref{thm:no-extra-factors} implies Theorems \ref{thm:LDS-in-RNP} and \ref{thm:le donne}.
\begin{proposition}\label{prop:bilip-decomp}
    Let $(X,\mu)$ be an LDS and $(U,\varphi)$ a $k$-dimensional Cheeger chart of $(X,\mu)$. Suppose $\Tan(X,\mu,x)=\{(V_x,c\leb^k,0)\}$ for $\mu$-a.e. $x\in U$, where $V_x$ is a $k$-dimensional Banach space. Then there is a countable disjoint collection $\{U_i\}$ with $\mu(U\setminus\bigcup_iU_i)=0$ so that $\Ha^k|_{U_i}\ll \mu|_{U_i}\ll \Ha^k|_{U_i}$ and $\varphi|_{U_i}$ is bi-Lipschitz.
\end{proposition}

\begin{proof}
Let $(U,\varphi)$ be as in the claim. We claim that 
\begin{align}\label{eq:liminf}
    \liminf_{y\to x}\frac{|\varphi(x)-\varphi(y)|}{d(x,y)}>0\quad\mu\textrm{-a.e. }x\in U.
\end{align}
To see this, let $A\subset U$ be the set of points where \eqref{eq:liminf} fails. For $\mu$-a.e. $x\in A$, every $(\R^k,c\leb^k,0,\psi)\in \Tan(X,\mu,x,\varphi)$ is such that $\psi:\R^k\to (\R^k,\|\cdot\|_x)$ is a submetry \cite[Corollary 7.99]{che-kle-sch16}. Since the domain and target have the same dimension, $\psi$ must be an isometry. For such $x\in A$, let $y_j\to x$ be such that
\begin{align*}
\lim_{j\to\infty}\frac{|\varphi(y_j)-\varphi(x)|}{r_j}=0,
\end{align*}
where $r_j=d(y_j,x)$. The sequence of rescalings $$(r_j\inv X,c(x,r_j)\inv\mu,x,r_j\inv(\varphi-\varphi(x)))$$ pmGH-subconverges to $(\R^k,c\leb^k,0,\psi)\in \Tan(X,\mu,x,\varphi)$ with realizations $\iota_j:X_j\to Z$, $\iota:\R^k\to Z$, and $\iota_j(y_j)\to \iota(y)\in \iota(\R^k)$ with $\|y\|=1$. Thus 
\begin{align*}
    \psi(y)=\lim_{j\to\infty}\frac{|\varphi(y_j)-\varphi(x)|}{r_j}=0,
\end{align*}
which contradicts the fact that $\psi$ is an isometry, and proves \eqref{eq:liminf}.

Now \eqref{eq:liminf} yields that the sets 
\begin{align*}
E_{m,n}=\{x\in U:\ |\varphi(y)-\varphi(x)|\ge d(y,x)/n\ \forall y\in B(x,1/m)\}
\end{align*}
cover $U$ up to a $\mu$-null set. Consequently $\varphi$ is bi-Lipschitz on $E_{m,n}\cap B(x,1/(2m))$ for every $x\in E_{m,n}$. Standard measure theoretic arguments and decomposition into sets with small diameter \cite[Proposition 3.1.1]{kei04'}, or \cite[Lemma 15.13]{mat95},  yield a countable disjoint collection $\{U_i\}$ of subsets with $\mu(U\setminus \bigcup_iU_i)=0$ so that $\varphi|_{U_i}$ is bi-Lipschitz. See also \cite[Corollary 3.2]{lytopen} for a similar argument. Since $(U_i,\varphi)$ is a $k$-dimensional Cheeger chart, we have that $\varphi_\ast(\mu|_{U_i})\ll \leb^k$, cf. \cite[Theorem 1.1]{de2017conjecture}. This and the fact that $\varphi|_{U_i}$ is bi-Lipschitz yield that $\Ha^k|_{U_i}\ll\mu|_{U_i}\ll \Ha^k|_{U_i}$. 
\end{proof}

\begin{proof}[Proof of Theorem \ref{thm:LDS-in-RNP}]
Let $(X,\mu)$ be an LDS and $\iota:X\to Y$ a bi-Lipschitz embedding. 
Then $(X',\mu')=(\iota(X),\iota_\ast\mu)\subset Y$ is an LDS. To prove the claim it suffices to show that $(X',\mu')$ is countably rectifiable. That is, we may assume $(X,\mu)\subset Y$. 

Let $(U,\varphi)$ be a $k$-dimensional Cheeger chart in $(X,\mu)$. By Theorem \ref{thm:no-extra-factors} we have that $\Tan(X,\mu,x)=\{(T_\mu(x),c\leb^k,0)\}$ for $\mu$-a.e. $x\in U$. Proposition \ref{prop:bilip-decomp} now implies that $U$ is $k$-rectifiable. Thus $X$ is countably rectifiable.
\end{proof}

We next present the proof of Theorem \ref{thm:le donne}.

\begin{proof}[Proof of Theorem \ref{thm:le donne}]
Let $(X,\mu)$ be an LDS. By Theorem \ref{thm:no-extra-factors} we have $\mu$-a.e. $x$ that $\Tan(X,\mu,x)=\{(T_\mu(x),c\leb^n,0)\}$ where $n=\dim T_\mu(x)$. Thus uniqueness of the limit implies that $r\inv(X-x)$ pGH converges to $T_\mu(x)$ $\mu$-a.e. $x$. 

Next, assume that $r\inv(X-x)$ pGH-converges to $T_\mu(x)$. By Theorem \ref{thm:tang=prod} and Lemma \ref{lem:e-close-to-product} the sets $r\inv Z_{rR}$ in Lemma \ref{lem:e-close-to-product} pGH-converge to a singleton. Thus they Hausdorff converge to a singleton. Since $r\inv(X-x)\cap B(R)$ is Hausdorff-close to $(r\inv Z_{rR}+T_\mu (x))\cap B(R)$, it follows that $r\inv(X-x)$ Hausdorff converges to $T_\mu(x)$ for $\mu$-a.e. $x\in X$.

Finally, if $r\inv(X-x)$ Hausdorff converges to $T_\mu(x)$, then by Theorem \ref{thm:tang=prod} we have $\Tan(X,\mu,x)=\{(T_\mu(x),c\leb^n,0)\}$ $\mu$-a.e. $x$. Proposition \ref{prop:bilip-decomp} then yields countable rectifiability. Since $\mu$ vanishes on porous subsets of $X$, $(X,\mu)$ is an LDS. This completes the proof.
\end{proof}

The remainder of this section is devoted to the proof of Theorem \ref{thm:no-extra-factors}, which involves the construction of non-differentiable functions in the presence of extra factors in tangents.

\subsection{Constructing non-differentiable functions}\label{sec:non-diff}

In our argument, it is convenient to use Schioppa's tiles (\cite[Definition 4.2]{sch16a}), but one could alternatively use ideas of Bate \cite{bate12diff} or Alberti-Csörnyei-Preiss \cite{acp}, where the idea of -- and techniques for -- constructing non-differential functions originally appeared. Throughout the remainder of this section, we assume that $(X,\mu)$ is an LDS, $X\subset B$ for some RNP-Banach space $B$, and $(U,\varphi)$ is an $n$-dimensional Cheeger chart of $(X,\mu)$.

The next definition is a minor modification of \cite[Definition 4.2]{sch16a}, see Remark \ref{rmk:tiles}.

\begin{definition}\label{def:tiles}
Let $\alpha,L,\varepsilon,C, r>0$. A $(L,\alpha,\varepsilon,r)$-tile with parameter $C$ at $x\in X$ is a pair $(S_x,f_x)$ where 
\begin{itemize}
	\item[(SET)] $S_x\subset  B(x,r)$ is a closed set with $\diam(S_x)\ge \frac rC$ and $\mu(S_x)\ge \frac {\mu(B(x,r))}C$
	\item[(FUN)]  $f_x:X\to \R$ is an $L$-Lipschitz function with $|f_x|\le L\dist(X\setminus S_x,\cdot)$
\end{itemize} 
and the following conditions are satisfied.
\begin{itemize}
	\item[(T1)] $\mu(S_x\cap \{|Df_x|>\varepsilon\})\le \varepsilon \mu(S_x)$
	\item[(T2)] there is $S_{x,var}\subset S_x$ with $\mu(S_x\setminus S_{x,var})\le \varepsilon\mu(S_x)$ such that, for every $y\in S_{x,var}$ there exists $y_{var}\in B(y,\varepsilon)\setminus\{y\}$ so that
	\[
	|f_x(y)-f_x(y_{var})|\ge \alpha d(y,y_{var}).
	\]
\end{itemize}
\end{definition}

\begin{remark}\label{rmk:tiles}
Definition \ref{def:tiles} corresponds to Schioppa's $(L,\alpha,\varepsilon,\varepsilon,r)$-tiles (\cite[Definition 4.2]{sch16a}) with the implied constant there being $C$, except for one difference in (SET), and explicating the constant $C$ in the definition. 

For Schioppa the constant $C$ is implicit and is allowed to depend on $L$. For us, in the application below, we will allow the constant $C$ to depend on $\alpha,L,\varepsilon$ instead of only on $L$. 

Since Schioppa's version of (SET) (\cite[Definition 4.2(T1) \& (T2)]{sch16a}) is only used in order to invoke Vitali's covering theorem -- which applies under our hypothesis (SET) -- the proof of \cite[Theorem 4.3]{sch16a} remains valid with our notion of tiles. 
\end{remark}

 The reason we seek to construct tile functions is that they are a simple conceptual tool to understand non-differentiability. Basically \cite[Theorems 4.1 and 4.3]{sch16a} state that if for some positive measure subset $U$ of $X$ one can construct tile functions at every point $x\in U$ with a sequence $r_n\searrow 0$, then $X$ is not a Lipschitz differentiability space. The tile functions here correspond to highly oscillating functions with very small derivative. The non-differentiable functions are obtained as infinite sums of such functions, where the smallness of the derivative guarantees that the resulting function behaves well.

 We denote by $\alpha(V,V')$ the Hausdorff-distance between the unit balls of two subspaces $V$ and $V'$ of $B$.

\begin{proposition}\label{prop:pre-tile} Fix parameters $\varepsilon,\delta, R,M>0$ with $\varepsilon < 1$ and an $n$-dimensional subspace $V_0\subset B$.
Let $K\subset U$ be a compact set for which 
\begin{align*}
    e_{B(R)}(T_\mu(x),r\inv(X-x))\stackrel{r\to 0}{\longrightarrow} 0\quad\mbox{uniformly in }K,
\end{align*}
and moreover for every $x\in K$:
\begin{itemize}
    \item[(i)] $\alpha(T_\mu(x),V_0)<\varepsilon^2$ for all $x\in K$;
    \item[(ii)] there exists $(Z+T_\mu(x),\nu_Z\times \leb^n,z_0)\in \Tan(X,\mu,x)$ with $\diam(Z)>0$.
    \item[(iii)]  $\mu(B(x,2r))\le M\mu(B(x,r))$ for all $x\in K$ and $0<r<\delta$.
\end{itemize}
Then for $\mu$-a.e. $x\in K$ there exists a sequence $r_j\downarrow 0$ so that for each $j$ the closed set $S_j:=\bar B(x+V_0,\varepsilon r_j)\cap \bar B(x,r_j)\cap X$ satisfies 
\begin{align}\label{eq:set}
\diam(S_j)\ge \frac 12 r_j\quad\mbox{and}\quad\mu(S_j)\ge c(\varepsilon,M)\mu(B(x,r_j))
\end{align}
and the function $g=g_{j,x}:=(\varepsilon r_j-\dist_{x+V_0})_+$ satisfies
\begin{itemize}
    \item[(a)] $\underset{B(y,2\varepsilon r_j)}{\operatorname{osc}}g\ge \frac{\varepsilon r_j}{4}$ for every $y\in S_j$; 
    \item[(b)] $|Dg|_\ast\le \varepsilon^2$ $\mu$-a.e. on $S_j\cap K$.
\end{itemize}
\end{proposition}

\begin{proof}
We may assume $\mu(K)>0$, otherwise the claim is vacuously true. Let $x$ be a Lebesgue density point of $K$, and a point of non-porosity  
 of $K$, so that $\Tan(X,\mu,x)=\Tan(K,\mu|_K,x)$, see e.g. \cite[Lemma 5.9 and Proposition 5.13]{soultaniscakovic}. Let $r_0>0$ be such that 
\begin{align*}
    e_{B(R)}(T_\mu(x),r\inv(X-x))\le \varepsilon^2,\quad 0<r\le r_0\mbox{ and }x\in K.
\end{align*}
Since $V_0\subset \bar C(T_\mu(x),\varepsilon^2)$, it follows that
\begin{align}\label{eq:V_0}
e_{B(R)}(V_0,r\inv(X-x))\le (1+R)\varepsilon^2, \quad 0<r\le r_0\mbox{ and }x\in K.
\end{align}

Choose by Theorem \ref{thm:tang=prod} $(Z+T_\mu(x),\nu_Z\times \leb^n,0)\in \Tan(X,\mu,x)=\Tan(K,\mu|_K,x)$ for which there exists $z_1\in Z$ with
\begin{align}\label{eq:rescale-to-e}
\dist(T_\mu(x),z_1+T_\mu(x))=\varepsilon.
\end{align}
Indeed, rescaling the tangent given by (ii) we can assume that \eqref{eq:rescale-to-e} holds. It follows from \eqref{eq:V_0} and \eqref{eq:rescale-to-e}, Lemma \ref{lem:e-close-to-product}, and the definition of tangents that there exist sequences $r_j\downarrow 0$ (realizing the convergence of the tangent) and  $y_j\subset K$ with 
\begin{align*}
\frac{d(y_j,x)}{r_j}\in [\varepsilon-R\varepsilon^2,\varepsilon+R\varepsilon^2],\quad\varepsilon-R\varepsilon^2\le \frac{d(y_j+w,x+v)}{r_j},\quad  v,w\in V_0\cap B(R)
\end{align*}

Moreover by \eqref{eq:V_0} we have that 
\begin{align*}
&(x+V_0)\cap B(x,r_jR)\subset B(X,(1+R)\varepsilon^2 r_j),\\
&(y_j+V_0)\cap B(x,r_jR)\subset B(X,(1+R)\varepsilon^2 r_j).
\end{align*}
Note that the function $g=(\varepsilon r_j-\dist_{x+V_0}(\cdot))_+$ defined in the claim satisfies
\begin{align*}
&g(x+v)=g(x)=\varepsilon r_j,\\
&g(y_j+v)=g(y_j)\in [0,R\varepsilon^2r_j], \quad v\in V_0
\end{align*}

Recall that $S_j=\bar B(x,r_j)\cap \bar B(x+V_0,\varepsilon r_j)\cap X$. For every $y\in S_j$ there exists $v_y,w_y\in V_0$ so that $x+v_y$ is a closest point on $x+V_0$ to $y$ and $y_j+w_y$ is a closest point on $y_j+V_0$ to $y$. In particular $\dist_{x+V_0}(y)=\|y-(x+v_y)\|\le \varepsilon r_j$. Let $x',y'\in X$ be such that 
\begin{align*}
\|x'-(x+v_y)\|\le 2R\varepsilon^2r_j,\quad \|y'-(y_j+w_y)\|\le 2R\varepsilon^2r_j.
\end{align*}
Then we have 
\begin{align*}
\dist_{x+V_0}(x')\le 2R\varepsilon^2 r_j,\quad \dist_{x+V_0}(y')\ge \varepsilon r_j-3R\varepsilon^2 r_j.
\end{align*}
If $d(y,x+V_0)\ge \varepsilon r_j/2$, then
\begin{align*}
|g(y)-g(x')|\ge g(x')-g(y)\ge \varepsilon r_j-\frac{\varepsilon r_j}{2}-  3R\varepsilon^2r_j,\quad d(y,x')\le \varepsilon r_j+2R\varepsilon^2r_j
\end{align*}
while, if $d(y,x+V_0)\le \varepsilon r_j/2$, then
\begin{align*}
|g(y)-g(y')|\ge g(y)-g(y')\ge \frac{\varepsilon r_j}{2}-3R\varepsilon^2r_j,\quad d(y,y')\le \varepsilon r_j+5R\varepsilon^2r_j.
\end{align*}
Indeed, 
\begin{align*}
    d(y,y')&\le \|y-(x+v_y)\|+\|x+v_y-(y_j+w_y)\|+\|y_j+w_y-y'\|\\
    &\le 2R\varepsilon^2r_j+\varepsilon r_j+R\varepsilon^2r_j+2R\varepsilon^2r_j=\varepsilon r_j+5R\varepsilon^2r_j.
\end{align*}
Let $\varepsilon$ be small enough so that $5R\varepsilon^2\le \varepsilon/4$. The estimates above yield
\begin{align}\label{eq:osc-estimate}
\underset{B(y,2\varepsilon r_j)}{\operatorname{osc}}g\ge \frac{\varepsilon r_j}{4},
\end{align}
which implies (a). 

To see (b), note that 
\begin{align*}
|(g\circ\gamma)_t'|\le\lim_{\dom(\gamma)\ni s\to t}\frac{|\dist_{x+V_0}(\gamma_s)-\dist_{x+V_0}(\gamma_t)|}{|s-t|}\le \varepsilon^2|\gamma'_t|\;\mbox{ a.e. }t\in \gamma\inv(K)
\end{align*}
since $\alpha(V_0,T_\mu(x))\le \varepsilon^2$, $x\in K$. It follows from this that $|Dg|_\ast\le \varepsilon^2$ $\mu$-a.e. on $K$, establishing (b).

It remains to prove the claimed properties of the set $S_j$. By \eqref{eq:V_0} we have $\diam(S_j)\ge r_j/2$.  
For all $j$ we have
\begin{align}\label{eq:S_j-size-est}
B(x,\varepsilon r_j/4)\cap X\subset  \bar B(x,r_j)\cap \bar B(x+V_0,\varepsilon r_j)\cap X=S_j.
\end{align}
The inclusions \eqref{eq:S_j-size-est} and condition (iii) yield the estimate
\begin{align*}
c(\varepsilon,M)\le \liminf_{j\to \infty}\frac{\mu(B(x,\varepsilon r_j/4))}{\mu(\bar B(x,r_j))}\le\liminf_{j\to\infty}\frac{\mu(S_j)}{\mu(\bar B(x,r_j))}
\end{align*}
from which the estimate
\[
\mu(S_j)\ge c(\varepsilon,M)\mu(\bar B(x,r_j))
\]
follows, completing the proof.
\end{proof}

For the next corollary, we denote by $\eta_\varepsilon:\R\to \R$ the continuous piecewise linear function with $\eta_\varepsilon|_{(-\infty,1-\varepsilon]}=1$ and $\eta_\varepsilon|_{[1,\infty)}=0$.
\begin{corollary}\label{cor:pre-tile}
Under the hypotheses of Proposition \ref{prop:pre-tile}, for $\mu$-a.e. $x\in K$ there exists a sequence $r_j\downarrow 0$ so that $(S_j,f_j)$ is a $(1/8,2,C\varepsilon,r_j)$-tile, where
\begin{align*}
f_j(y)=\eta_\varepsilon\Big(\frac{d(y,x)}{ r_j}\Big)g(y)
\end{align*}
and the constant $C$ is independent of $\varepsilon$ and $j$ but depends on $n$, $M$. Here $g$ is as in Proposition \ref{prop:pre-tile}.
\end{corollary}
\begin{proof}
Let $x\in K$ be a Lebesgue density point of $K$ such that the conclusion of Proposition \ref{prop:pre-tile} holds.  By Theorem \ref{thm:tang=prod}, for each $x$ we may assume that the tangent $Z+T_\mu(x)$ in condition (ii) of Proposition \ref{prop:pre-tile} is chosen so that $Z\subset \ker(p_x)$ for a projection $p_x:B^\omega\to T_\mu(x)$ with $\|p_x\|_{op}\le n$. By \eqref{eq:set} the set $S_j=\bar B(x,r_j)\cap \bar B(x+V_0,\varepsilon r_j)$ satisfies (SET) in Definition \ref{def:tiles}. Since $g$ is 1-Lipschitz and $0\le g\le \varepsilon r_j$, it follows that $f_j$ is 2-Lipschitz. Moreover $f_j=0$ outside $S_j$. Consequently $f_j$ satisfies (FUN) with $L=2$.

To show (T1) and (T2) we note that $f_j=(\varepsilon r_j-\dist_{x+V_0})_+=g$ on $\bar B(x,r_j-\varepsilon r_j)$. Denote 
\[
S_{j,var}:=\bar B(x,r_j-3\varepsilon r_j)\cap S_j,\quad E_j:=B(x+V_0,\varepsilon r_j)\cap (\bar B(x,r_j)\setminus \bar B(x,r_j-3\varepsilon r_j)),
\]
so that $S_{j,var}\cap E_j=\varnothing$ and $S_j=S_{j,var}\cup E_j$. Note that $S_j\cap \{|Df_j|_\ast>\varepsilon\}\subset E_j\cup(\bar B(x,r_j)\setminus K)$ by Proposition \ref{prop:pre-tile}(b). 
Moreover by Proposition \ref{prop:pre-tile}(a), for all $y\in S_{j,var}$ there exists $y_{var}\in B(y,\varepsilon)\setminus \{y\}$ with 
\[
|f(y)-f(y_{var})|\ge d(y,y_{var})/8.
\]
 Since $x$ is assumed to be a Lebesgue density point of $K$ we have $\frac{\mu(\bar B(x,r_j)\setminus K)}{\mu(S_j)}\stackrel{j\to\infty}{\longrightarrow}0$ and thus, to establish (T1) and (T2), it suffices to show that $\mu(E_j)\lesssim_{n,M} \varepsilon \mu(S_j)$ for large $j$. Suppose $\phi_j:r_j\inv(X-x)\to Z+T_\mu(x)$ are $(R_j,\varepsilon_j)$-isometries with $\varepsilon_j\downarrow 0$ and $R_j\uparrow \infty$ such that $\phi_{j\ast}\mu_j\rightharpoonup\nu_Z\times\leb^n$. For large enough $j$ we have that
\begin{align*}
    \phi_j(r_j\inv(E_j-x))&\subset (B_Z(0,C\varepsilon)+T_\mu(x))\cap (B(0,1+C\varepsilon)\setminus B(0,1-C\varepsilon))\\
    &\subset B_Z(0,C\varepsilon)+B_{T_\mu(x)}(0,1+C\varepsilon)\setminus B_{T_\mu(x)}(0,1-C\varepsilon):=A_\varepsilon,
\end{align*}
where $C$ depends only on $n$. From this inclusion and the weak convergence $\phi_{j\ast}\mu_j\rightharpoonup \nu_Z\times\leb^n$ (recall that $\mu_j=\frac{\mu}{c_\mu(x,r_j)}$) we obtain
\begin{align*}
\limsup_{j\to\infty}\frac{\mu(E_j)}{c_\mu(x,r_j)}&\le \limsup_{j\to\infty}\mu_j(\phi_j\inv(A_\varepsilon))\le \nu_Z\times\leb^n(A_\varepsilon)\\
&= \omega_n\nu_Z(B_Z(0,2\varepsilon))[(1+C\varepsilon)^n-(1-C\varepsilon)^n]\\
&  \lesssim_{n}\ \varepsilon\nu_Z(B_Z(0,2\varepsilon)) \lesssim_{n,M}\ \varepsilon\nu_Z( B_Z(0,\varepsilon/2))\\
&= \tilde c(n,M)\varepsilon \frac{(\nu_Z\times\leb^n)( B_Z(0,\varepsilon/2)+T_\mu(x)\cap  B(0,1-\varepsilon^2))}{\leb^n(T_\mu(x)\cap  B(0,1-\varepsilon^2))}\\
&\le c(n,M)\varepsilon (\nu_Z\times\leb^n)(B_Z(0,\varepsilon/2)+T_\mu(x)\cap  B(0,1-\varepsilon^2))
\end{align*}
 for $\varepsilon<1/2$. Here the equality in the second to last line follows from the fact that $\nu_Z\times \leb^n(A+B)=\nu_Z(A)\leb^n(B)$ for $A\subset Z$, $B\subset T_\mu(x)$.

The inclusion \eqref{eq:S_j-size-est} and the convergence $\phi_{j\ast}\mu_j\rightharpoonup \nu_Z\times\leb^n$ yield
\begin{align*}
&(\nu_Z\times\leb^n)( B_Z(0,\varepsilon/2)+T_\mu(x)\cap  B(1-\varepsilon^2))\\
\le &\lim_{j\to\infty}\phi_{j\ast}\mu_j(B_Z(0,\varepsilon/2)+T_\mu(x)\cap  B(1-\varepsilon^2))\le\liminf_{j\to\infty}\frac{\mu(S_j)}{c_\mu(x,r_j)}
\end{align*}
which, together with the previous estimate yields $\mu(E_j)\lesssim_{n,M,x} \varepsilon \mu(S_j)$ for large $j$, as required.
\end{proof}

We conclude this section with the proof of Theorem \ref{thm:no-extra-factors}.

\begin{proof}[Proof of Theorem \ref{thm:no-extra-factors}]
Suppose by contradiction that there exists a set $A\subset U$ with $\mu(A)>0$ so that for each $x\in A$, there exists a tangent $(Z+T_\mu(x),\nu_Z\times\leb^n,0)\in \Tan(X,\mu,x)$ with $\diam(Z)>0$.  By \cite[Lemma 2.2. and 2.3]{bateli} we may assume that $\mu$ is $(M,\delta)$-doubling along $A$, that is Proposition \ref{prop:pre-tile}(iii) holds. Arguing as in the proof of Lemma \ref{lem:e-close-to-product}\eqref{eq:e2} we find a compact set $K\subset A$ of positive $\mu$-measure and an $n$-plane $V_0\subset Y$ such that
\begin{align*}
    e_{B(R)}(T_\mu(x),r\inv(X-x))\stackrel{r\to 0}{\longrightarrow} 0\quad\mbox{uniformly in }K,
\end{align*}
and $\alpha(T_\mu(x),V_0)<\varepsilon^2$ for all $x\in K$. Thus we may apply Proposition \ref{prop:pre-tile} and Corollary \ref{cor:pre-tile} to obtain that for every $\varepsilon>0$, $\mu$-a.e. $x\in K$ there exists a sequence $r_j\downarrow 0$ and $(2,1/8,\varepsilon,r_j)$-tiles $(S_j,f_j)$ at $x$. By Theorem \cite[Theorem 4.1 and Theorem 4.3]{sch16a}, $(K,\mu|_K)$ is not a differentiability space, which is a contradiction. Thus $\mu(A)=0$, completing the proof of the theorem.
\end{proof}

\appendix
\section{Blow-ups of Alberti representations}
We assume that $X\subset B$ is a complete separable space, where $B$ is an RNP-Banach space, and $\mu$ is a boundedly finite measure on $X$.
\begin{proposition}\label{prop:const z-component}
Let $(U,\varphi)$ be an $n$-dimensional fragment-wise chart of $(X,\mu)$. Then $\mu$-a.e. $x\in U$ the following holds: if $(Z+T_\mu(x),\nu,0)\in \Tan(X,\mu,x)$ and $\mathcal A_\infty=[Q_\infty,w_\infty]$ is a blow-up of an Alberti representation $\mathcal A=[Q,w]$ in the sense of \cite[Theorem 7.18]{CKS} or \cite[Theorem 3.11 and Corollary 3.14]{sch16a}, then we have that $p_Z\circ\gamma$ is constant $Q_\infty$-a.e. $\gamma$.
\end{proposition}

We will adapt a statement in \cite{CKS} by using a well-chosen container for the pointed measured Gromov-Hausdorff convergence. We expect that there are also other proofs of this proposition. The key observation is that for any Alberti representation $\gamma'_t\in T_\mu(\gamma_t)$ for a.e. $t\in \dom(\gamma)$. Thus, the component of $\gamma(t+h)-\gamma(t)$ that is tranverse to $T_\mu$ has vanishing norm infinitesimally. Blowing up such a representation, as was done in \cite{CKS}, one obtains a representation where the curve fragments become geodesic lines, and the transverse components of their velocity vectors vanish. Instead of redoing the full argument along these lines, it is however easier to use the statements in \cite{CKS} with small edits.

\begin{proof}
The argument in the proof of Theorem \ref{thm:tang=prod}(i) yields the following. In the notation of that proof, for any $R>0$ and  $y\in K:=K_m^j\cap B(x,rR)$  we have that
\begin{align*}
\frac{y-x}{r}+T_\mu(x)\cap B(R)\subset B(r\inv(K\cap B(x,rR)-x),2\varepsilon)
\end{align*}
Let $p_x:B\to T_\mu(x)$ be a projection and set $Z_\rho=(\id-p_x)(K\cap B(x,\rho)-x)$. It follows that 
\begin{align*}
(r\inv Z_{rR}+T_\mu(x))\cap B(R)\subset B(r\inv(K\cap B(x,rR)-x),2\varepsilon).
\end{align*}
Thus, $r\inv(K\cap B(x,rR)-x)$ is $2\varepsilon$-close to a product (sum) $r\inv  Z_{rR}+T_\mu(x)$. Note that $r\inv(K\cap B(x,rR)-x)\subset r\inv Z_{rR}+T_\mu(x)$ as well.

Now let $r_j\downarrow 0$ be a sequence of scales realizing the convergence of some $(Z+T_\mu(x),\nu,0)\in \Tan(X,\mu,x)$.  Let $R_j\uparrow\infty$  and $\varepsilon_j\downarrow 0$ be such that there are $(R_j,\varepsilon_j)$-isometries 
\begin{align*}
\psi_j:r_j\inv(K-x)\to Z+T_\mu(x)
\end{align*}
realizing the convergence $r_j\inv(K-x)\stackrel{pmGH}{\longrightarrow}Z+T_\mu(x)$. Then $Z_j:=r_j\inv Z_{r_jR_j}\stackrel{pGH}{\longrightarrow}Z$ and $Z_j+T_\mu(x)\stackrel{pGH}{\longrightarrow}Z+T_\mu(x)$.  
Indeed, the latter convergence follows from the fact that $r_j\inv(K-x)\stackrel{pGH}{\longrightarrow}Z+T_\mu(x)$ and $Z_j+T_\mu(x)$ is Hausdorff-close to $r_j\inv(K\cap B(x,r_jR_j)-x)$ (in the ambient space $B$). Since $Z_j$ is a subset of $Z_j+T_\mu(x)$ (which pGH converges), it pGH-subconverges to a limit $Z'$, and since then $Z_j+T_\mu(x)\stackrel{GH}{\longrightarrow}Z'+T_\mu(x)$, it must be that $Z=Z'$. Repeating this argument for an arbitrary subsequence we conclude that $Z_j\stackrel{pGH}{\longrightarrow}Z$. 

Recall that $r_j\inv(K\cap B(x,r_jR_j)-x)\subset Z_j+T_\mu(x)$. Considering  (the restriction of) $\mu_j$ (to $K\cap B(x,r_jR_j)$) as a measure on $Z_j+T_\mu(x)$, we have
\begin{align*}
(Z_j+T_\mu(x),\mu_j,0)\stackrel{pmGH}{\longrightarrow} (Z+T_\mu(x),\nu,0).
\end{align*}

We claim that the blow-up of any Alberti representation of $\mu|_K$ is concentrated on lines with constant $Z$-component. This follows from the proof of \cite[Theorem 3.11]{sch16a} and/or \cite[Theorems 7.18 and 7.22]{CKS}.\footnote{As remarked in \cite{sch16a}, the conclusions of Theorems 7.18 and 7.22 in \cite{CKS} hold without the assumption that the space is LDS. Indeed, it suffices to assume that $\mu$ is pointwise doubling and admits an Alberti representation.} We explain the additional steps in the proofs therein to establish constancy of the $Z$-component. 

Denote $Q=\id-p_x$, where $p_x:B\to T_\mu(x)$ is a given projection map (note that this extends canonically to a projection $p_x^\omega:B^\omega\to T_\mu(x)$ on the ultrapower $B^\omega$, see the discussion in the beginning of Section \ref{sec:RNP}), with operator norm $\|Q\|$. Let $W$ and $\widehat W$ be Banach space containers,\footnote{Terminology from Schioppa \cite[Definition 3.1]{sch16a}.} and $\phi_j:Z_j\to W$, $\widehat\phi_j:Z_j+T_\mu(x)\to \widehat W$ the isometric embeddings, realizing the convergence $Z_j\stackrel{pGH}{\longrightarrow}Z$ and $(Z_j+T_\mu(x),\mu_j,0)\stackrel{pmGH}{\longrightarrow}(Z+T_\mu(x),\nu,0)$, respectively ($j\in\N\cup\{\infty\}$). Here, we set $Z=Z_\infty$ for notational convenience.

Consider $\widetilde W=W\times \widehat W$ equipped with metric $\tilde d=\max\{\|Q\|\inv d_W,d_{\widehat W}\}$ (if $W,\widehat W$ are normed spaces, $\tilde d$ is induced by a norm), and define an isometric embedding
\begin{align*}
\iota_j:Z_j+T_\mu(x)\to \widetilde W,\quad (z,v)\mapsto (\phi_j(z),\widehat\phi_j(z,v))
\end{align*}
for $j\in\N\cup\{\infty\}$. Indeed, the isometry of $\iota_j$ follows from the estimate
\begin{align*}
d_W(\phi_j(z),\phi_j(z'))=\|z-z'\|\le \|Q\|\|z+v-(z'-v')\|,\quad z,z'\in Z_j,\; v,v'\in T_\mu(x),
\end{align*}
and the definition of the metric $\tilde d$. It is straightforward to check that $\widetilde W$ is also a container for the convergence $(Z_j+T_\mu(x),\mu_j,0)\stackrel{pmGH}{\longrightarrow}(Z+T_\mu(x),\nu,0)$. Consider the lower semicontinuous functional 
\begin{align*}
    \ell_W^*:\Fr(\widetilde W)\to \R,\quad \ell_W^*(\gamma)=\ell^*(p_W\circ\gamma)
\end{align*}
Here $\ell^*(\gamma)=\Ha^1(\im(\gamma))+\gap(\gamma)$. 

Now, following the proof of \cite[Theorem 3.11]{sch16a} or \cite[Theorems 7.18 and 7.22]{CKS} we obtain measures $Q_j$\footnote{Denoted $P_j$ in \cite{CKS}}. By the weak convergence $Q_j\stackrel{w^*}{\rightharpoonup} Q_{R_0}$ and the lower semicontinuity of $\ell_W^*$ we have
\begin{align*}
   \|Q\|\inv \int_{U(R)}\ell^*(p_Z\circ\gamma)\ud Q_{R_0}=\int_{U(R)}\ell^*_W\ud Q_{R_0}\le \liminf_{j\to \infty}\int_{U(R)}\ell^*_W\ud Q_j
\end{align*}
for any $R$, where $U(R)=\{\gamma\in Fr(\widetilde W):\ \im(\gamma)\subset B(w,R)\}$ is an open set in $\Fr(\widetilde W)$. 

We recall that $Q_j$ is concentrated on the set of fillings of curves in $X_j$, and thus
\begin{align}\label{eq:small1}
    |(\gamma_{Z_j})'_t|=|(Q\circ\gamma)'_t|\le \|Q\|\varepsilon^2\quad\mbox{a.e. }\; t\in \gamma\inv(X_j).
\end{align}
Here, we used the fact that $\gamma_t'\in V_{\gamma_t}\subset \bar C(T_\mu(x),\varepsilon^2)$ for a.e. $t\in \dom(\gamma)$. Moreover, for $Q_j$-a.e. $\gamma$ we have by {\bf (Reg5)} in \cite[Section 7]{CKS} that 
\begin{align}\label{eq:small2}
    \gap(\gamma)\le 2D_0\varepsilon_m.
\end{align}
Together \eqref{eq:small1} and \eqref{eq:small2} imply 
\begin{align*}
    \ell_W^*(\gamma)\le \|Q\|\varepsilon^2\Ha^1(\im(\gamma))+D_0\varepsilon_m\le \varepsilon^2\cdot 2R_0\|Q\|+2D_0\varepsilon_m\quad Q_j-\mbox{a.e. }\; \gamma
\end{align*}
In particular 
\begin{align*}
    \int_{U(R)}\ell^*(p_Z\circ\gamma)\ud Q_{R_0}(\gamma)\le \liminf_{j\to\infty}\int \ell_W^*\ud Q_j\lesssim (\varepsilon^2+\varepsilon_m)Q_j(U(R))
\end{align*}
for all $R>0$. Note that the total mass of $Q_j$ is uniformly bounded, cf. \cite[(7.78)]{CKS}. Since $\varepsilon>0$ is arbitrary and $\sum_m\varepsilon_m<\infty$, letting $\varepsilon\to 0$ and $m\to \infty$ we obtain 
\begin{align*}
\ell^*(p_Z\circ\gamma)=0\quad Q_{R_0}\mbox{-a.e. }\gamma
\end{align*}
The measure $Q_\infty$ obtained in \cite[Theorem 3.11 and Corollary 3.14]{sch16a} and \cite[Theorem 7.18]{CKS} is a limit of the measures $Q_{R_0}$ as $R_0\to\infty$ -- cf. \cite[Lemma 7.85]{CKS} and thus 
\begin{align*}
\ell^*(p_Z\circ\gamma)=0\quad Q_\infty-\mbox{a.e. }\; \gamma.
\end{align*}
This implies that $p_Z\circ \gamma$ is constant $Q_\infty$-a.e., as claimed.
\end{proof}

\bibliographystyle{plain}
\bibliography{abib}
\end{document}